\documentclass[11pt]{amsart}

\usepackage[margin=1in]{geometry}
\usepackage{amsmath,amssymb,amsfonts}
\usepackage{amsthm}
\usepackage{mathrsfs}
\usepackage[utf8]{inputenc}
\usepackage[T1]{fontenc}
\usepackage{orcidlink} 

\theoremstyle{plain}
\newtheorem{theorem}{Theorem}
\newtheorem{lemma}{Lemma}
\newtheorem{corollary}{Corollary}
\newtheorem{proposition}{Proposition}

\theoremstyle{definition}
\newtheorem{definition}{Definition}
\newtheorem{example}{Example}
\newtheorem{assumption}{Assumption}
\newtheorem{remark}{Remark}

\begin{document}

\title[Weighted Jackson-Bernstein estimates]{Jackson and Bernstein Inequalities with Explicit Weight Dependence for Band-Limited Wavelet Approximation in Weighted Besov Spaces}

\author{Hung-Lung Lin \orcidlink{0000-0001-8904-7750}}
\address{School of Economics and Management, Sanming University, No. 25, Jing Dong Road, Sanming City 365004, Fujian, China}
\email{hsa8936.hsa8936@msa.hinet.net}

\author{Kai-Cheng Wang \orcidlink{0000-0001-7322-3625}}
\address{School of Information Engineering, Sanming University, No. 25, Jing Dong Road, Sanming City 365004, Fujian, China}
\email{gtotony98@gmail.com}

\thanks{*Correspondence: gtotony98@gmail.com}

\subjclass[2020]{Primary 41A17; Secondary 41A25, 42C40, 42B25, 42B35, 46E35}
\keywords{Jackson-type inequality, Bernstein inequality, wavelet approximation, weighted Besov space, Muckenhoupt weight}

\begin{abstract}
We prove direct and inverse estimates for band-limited biorthogonal wavelet
systems in weighted Besov spaces with Muckenhoupt weights. On functions
recovered from their wavelet expansion along the dyadic tree of the unit
cube, the error of approximation by a prescribed number of terms decays at
a rate set by the smoothness-to-dimension ratio, and this exponent cannot
be increased. Under a multiresolution condition, Bernstein inequalities hold
at a single scale and on every multiresolution subspace, and the linear
approximation spaces are identified with weighted Lebesgue--Besov
intersections. All constants are explicit; the weight enters only
through displayed powers of the Muckenhoupt characteristic.
\end{abstract}

\maketitle

\section{Introduction}\label{sec1}
\everymath{\displaystyle}
The quantitative theory of approximation rests on two reciprocal
inequalities. A direct estimate of Jackson type bounds the error of
approximation by the smoothness of the target. An inverse estimate of
Bernstein type bounds the smoothness of an approximant by its size and by
the resolution at which it is assembled. When the two hold at a common
order, the machinery of approximation spaces converts them into a
description of the functions reachable at a prescribed rate below the common
order, and that description coincides with a smoothness class on the
underlying scale \cite{DeVore1993, DeVore1998}. Wavelet bases supply a
setting in which this conversion proceeds with particular economy, since
most function norms admit a description through wavelet coefficients
\cite{FrazierJawerth1990, Ronal1992}. The present paper develops both halves
of this pairing on the weighted Besov scale.

Nonlinear $N$-term wavelet approximation occupies a central role in modern
approximation theory. Given a function in a smoothness class, the question
is how accurately it may be approximated in a prescribed norm by a linear
combination of at most $N$ elements of a fixed wavelet system, the
selection of indices and the choice of coefficients being allowed to depend
on the function. For the target space $L^p(\mathbb{R}^n)$ the classical
answer of DeVore, Jawerth, and Popov \cite{Ronal1992} identifies the rate
$N^{-s/n}$ for the Besov class $B^s_{\tau}(L^\tau)$ on the critical line
$1/\tau = s/n + 1/p$, and structural and quantitative refinements appear in
\cite{DeVore1993, Temlyakov2011, GARRIGOS200870, SICKEL2009748, Dinh2018}.
When the smoothness class and the target space share the integrability
exponent $p$, the situation is different. The rate $N^{-s/n}$ is then tied
to the number of wavelet coefficients that each scale carries. For
functions recovered from their wavelet expansion along the dyadic tree of
the unit cube, whose generation $j$ consists of $2^{jn}$ cubes, the rate is
attained already by truncation in scale.

Once a Muckenhoupt weight $w$ is introduced, the relevant Lebesgue scale
becomes $L^p_w$ and the corresponding smoothness scale becomes the
homogeneous weighted Besov space $\dot{B}^s_{p,q}(w)$. The weight encodes local degeneracy of the measure against which size is
recorded, and it returns the unweighted theory at the constant weight
$w\equiv 1$. Structural results
in this direction were established by Bui et al.
\cite{BuiPaluszynskiTaibleson1996}, by Rychkov in the local Muckenhoupt
setting \cite{Vyachesla2001}, by Bownik in the anisotropic doubling-measure
framework \cite{Bownik2003, Bownik2005, Bownik2008}, and by Li et al. in the
weighted anisotropic Triebel--Lizorkin direction \cite{Li2011}. Parallel
developments employing variable-exponent and generalized scales appear in
\cite{Izuki2015, Izuki2020, KOGURE2024106037, kwok2016}. Boundedness of the $\varphi$-transform on anisotropic Besov spaces
with doubling measures is established in \cite[Theorem~3.5]{Bownik2005}. The Jackson-type estimate of the present paper passes from the Besov
norm to the wavelet coefficients only through the one-sided coefficient bound of Lemma~\ref{lem:coefficient-bound}, which
is proved directly with explicit constants.

A single feature organizes the estimates that follow, namely the explicit
dependence of every constant on the Muckenhoupt characteristic
$[w]_{A_p}$. The weighted theory of the present era measures the size of an
operator by the power of $[w]_{A_p}$ that it carries, the linear dependence
on the $A_2$ characteristic being the reference point fixed for
Calder\'on--Zygmund operators in \cite{Hytonen2012}. We adopt the same
accounting. Each single-scale passage, whether a synthesis, an analysis or a
convolution with a kernel concentrated at one dyadic scale, is assigned the
power $[w]_{A_p}^{1/p}$, and these powers are displayed rather than absorbed
into an unnamed constant. The accounting is what distinguishes a transparent
estimate, whose weight dependence is visible at every line, from one in
which that dependence is left implicit.

On the direct side, the present paper addresses the quantitative side of the weighted Jackson
problem when the smoothness class is $\dot{B}^s_{p,q}(w)$ and the target
space is $L^p_w$ with the same exponent $p$. On the whole of
$\mathbb{R}^n$ an inequality of the form
$E^{\mathcal{F}}_N(f)_{L^p_w} \le C N^{-s/n}\|f\|_{\dot{B}^s_{p,q}(w)}$ cannot be expected,
because a single scale of $\dot{B}^s_{p,q}(w)$ admits arbitrarily many
coefficients of comparable size. For $w\equiv 1$,
Remark~\ref{rem:localization-necessary} records an explicit family of this
kind, for which no inequality of this form holds. The inequality becomes available
for functions recovered in $L^p_w$ from their wavelet expansion along the
dyadic tree of the unit cube $Q_0=[0,1)^n$, namely along the indices
$(j,k)$ with $j\ge 0$ and $Q_{j,k}\subset Q_0$. For band-limited systems, whose generators have no
compact support, this localization at the level of coefficients is the
counterpart of approximation on a cube, and it is the setting of the main
result.

The main result, Theorem~\ref{thm:main1}, states that for every $f$ in the
localized class $\dot{B}^s_{p,q}(w;Q_0)$ of
Definition~\ref{def:localized-class} and every integer $N \ge 1$,
\begin{align}\label{eq:intro-jackson-bound}
	&E^{\mathcal{F}}_N(f)_{L^p_w}
	\le C_0\,[w]_{A_p}^{2/p}\,N^{-s/n}\,\|f\|_{\dot{B}^s_{p,q}(w)}
	\le C_0\,[w]_{A_p}^{2\theta(p)}\,N^{-s/n}\,\|f\|_{\dot{B}^s_{p,q}(w)},\\
	&\theta(p)= \tfrac{1}{2}\max\left\{1,\,(p-1)^{-1}\right\}.\nonumber
\end{align}
The factor $C_0$ is defined explicitly in \eqref{eq:constant-form}, and
Theorem~\ref{thm:main1} records a closed-form majorant of it. The factor $C_0$
depends on the dimension, the number of generators, the exponent $p$, the
smoothness index $s$, the finite-band index $M$ of \eqref{eq:finite-band}, which compares the Littlewood--Paley annulus with the frequency annulus of the generators, and two decay quantities of the synthesis and dual generators.
It does not depend on $N$, on $f$, on $q$, on $w$, or on the Riesz bounds of
the system. The weight enters only through the displayed power of
$[w]_{A_p}$, which is the product of two factors $[w]_{A_p}^{1/p}$. The
first arises in the synthesis of a single scale, and the second arises in
the passage from the Littlewood--Paley blocks of $f$ to its wavelet
coefficients. The second inequality in \eqref{eq:intro-jackson-bound}
follows from $[w]_{A_p}\ge 1$ and $2/p\le 2\theta(p)$. The exponents $2/p$ and $2\theta(p)$ coincide exactly
at $p=2$, where the two bounds in \eqref{eq:intro-jackson-bound} are identical.

On the inverse side, the wavelet system is assumed in addition to carry the
multiresolution structure of Assumption~\ref{ass:mra}, which supplies the
scaling functions from which the approximating subspaces $V_m$,
$m\in\mathbb{Z}$, of Definition~\ref{def:b-multiresolution} are formed;
Example~\ref{ex:meyer-mra} shows that the Meyer systems satisfy it. Three
estimates carry the inverse theory across the resolutions of the wavelet
system. The first, Theorem~\ref{thm:B1}, is a single-scale Bernstein
estimate. For a finite wavelet sum at one generation it controls the
weighted Besov quasi-norm by the weighted Lebesgue norm in both directions,
the lower control being free of the Muckenhoupt characteristic and the upper
control carrying the power $[w]_{A_p}^{1/p}$. The second,
Theorem~\ref{thm:B2}, is a Bernstein inequality on $V_m$. For an arbitrary
element of $V_m$ it bounds the weighted Besov quasi-norm by the weighted
Lebesgue norm and the scale factor $2^{ms}$, with the power
$[w]_{A_p}^{1/p}$; it rests on the vanishing of the high-frequency
Littlewood--Paley blocks on $V_m$, not on the single-scale estimate. The
third, Theorem~\ref{thm:B3}, identifies the approximation space generated by
the family $\{V_m\}$ with the weighted smoothness class, in the form
\[
\mathcal{A}^{s/n}_q(L^p_w, \{V_m\}) = L^p_w \cap \dot{B}^s_{p,q}(w),
\]
through a direct inequality with the power $[w]_{A_p}^{2/p}$ and an inverse
inequality with the power $[w]_{A_p}^{1/p}$. The matching lower estimates for
the weight powers of the direct and inverse inequalities, the nonlinear
strengthening on the critical line $1/\tau = s/n + 1/p$, and the
Marchaud-type companion to the present direct--inverse pair are left to
separate directions, as recorded in Remark~\ref{rem:Bern:B3-linear}.

\subsection{Originality: explicit weight powers in direct and inverse estimates}
\label{subsec:originality}

\medskip
\noindent\textbf{A cube-by-cube single-scale estimate.} The engine of the
argument is the two-sided estimate of Lemma~\ref{lem:single-scale-estimate}.
For a single-scale wavelet sum
$r^{(j)}_\Lambda=\sum_{k\in\Lambda}c_{j,k}\psi_{j,k}$ it compares
$\|r^{(j)}_\Lambda\|_{L^p_w}$ with the weighted coefficient quantity
$2^{jn/2}\left(\sum_{k\in\Lambda}|c_{j,k}|^p\,w(Q_{j,k})\right)^{1/p}$ in
both directions, with the powers $[w]_{A_p}^{\pm\theta(p)}$. The proof
given here proceeds cube by cube. It invokes the $A_p$ condition only
through the two inequalities for $A_p$ weights on cubes recorded in
Proposition~\ref{prop:ap-cube}, it uses the generators through their decay
and the biorthogonality relation, and it yields the powers
$[w]_{A_p}^{-1/p}$ and $[w]_{A_p}^{1/p}$ in the lower and upper estimates. Since $[w]_{A_p}\ge 1$ and $1/p\le\theta(p)$, one has $[w]_{A_p}^{-\theta(p)}\le[w]_{A_p}^{-1/p}$ and $[w]_{A_p}^{1/p}\le[w]_{A_p}^{\theta(p)}$, so the refined estimate implies the lemma (Remark~\ref{rem:single-scale-two-sided}). The
constants $K_U$, $K_L$ of \eqref{eq:sse-refined} are explicit series in the
decay envelopes of the generators, and \eqref{eq:sse-KU}, \eqref{eq:sse-KL}
bound them in closed form.

\medskip
\noindent\textbf{No implicit characterization constant.} The passage from
the Besov norm of $f$ to its wavelet coefficients, which is needed for the
upper bound, is carried out directly in Lemma~\ref{lem:coefficient-bound}
with the power $[w]_{A_p}^{1/p}$. Consequently the constant in
\eqref{eq:intro-jackson-bound} contains no factor inherited from the
equivalence constants of a cited characterization, and every appearance of
$[w]_{A_p}$ is traceable to a single application of
Proposition~\ref{prop:ap-cube}. More generally, every weighted bound that enters the statement of a lemma, theorem or corollary in
Sections~\ref{sec3} and~\ref{sec4} is a single-scale bound obtained cube by cube from
Proposition~\ref{prop:ap-cube}, and each such bound contributes the power
$[w]_{A_p}^{1/p}$. No extrapolation theorem, no quantitative bound for the
maximal function and no wavelet characterization of the weighted Besov space
is used. The construction is principled rather than calibrated, in the sense
that no constant is tuned to the weight, every power being read from the
structure of the operator that produces it.

\medskip
\noindent\textbf{A lower estimate and the role of the localization.} The
rate in \eqref{eq:intro-jackson-bound} is complemented by
Corollary~\ref{lem:lower-bound-constant}. For each $w\in A_p$, an explicit family of single-scale
wavelet sums $\{f^{*}_N\}_{N=2^{mn}}$ in the localized class, whose coefficients are normalized by the cube measures $w(Q_{j,k})$ of the generation $j=m+1$, satisfies
\begin{equation}\label{eq:intro-lower-bound}
	c_\flat\,[w]_{A_p}^{-2/p}
	\le
	\frac{E^{\mathcal{F}}_N(f^{*}_N)_{L^p_w}}
	{N^{-s/n}\,\|f^{*}_N\|_{\dot{B}^s_{p,q}(w)}}
	\le
	c_\sharp\,[w]_{A_p}^{2/p},
\end{equation}
with $c_\flat, c_\sharp$ independent of $N$ and $w$. Consequently the
exponent $s/n$ in \eqref{eq:intro-jackson-bound} cannot be replaced by a
larger one on the localized class. The construction is deterministic and
uses neither random signs nor Khinchine-type averaging. No claim is made
that the displayed power of $[w]_{A_p}$ is matched in order from below.
For $w\equiv 1$, Remark~\ref{rem:localization-necessary} shows that the
localization itself cannot be removed.

\medskip
\noindent\textbf{The weight profile of the inverse estimates.} The
single-scale Bernstein estimate is two-sided, and its two sides are not
symmetric in the weight. The lower control is read off one term of the
Littlewood--Paley resolution and reconstructed across the finite spectral
band of a band-limited generator, so it uses the triangle inequality in
$L^p_w$ and no further weighted device. For that reason it is free of
$[w]_{A_p}$, a profile that the unweighted theory does not make visible
because there the weight is absent from the outset. The upper control
passes each of the finitely many nonvanishing Littlewood--Paley blocks
through a weighted bound for the Littlewood--Paley convolution at the
corresponding scale, an integral operator whose kernel is concentrated at
a single dyadic scale, and the power $[w]_{A_p}^{1/p}$ of that bound is the
only factor that depends on the weight. For a single generator, the lower
control is recorded in Remark~\ref{rem:single-scale-two-sided}(d), where the
upper control is expressed through a weighted $\ell^p$ norm of the
coefficients rather than through the weighted Lebesgue norm. The Bernstein
inequality on $V_m$ carries the same power $[w]_{A_p}^{1/p}$ as the
single-scale upper control. The Fourier transforms of the scaling functions
have compact support, so that the Littlewood--Paley blocks whose index
exceeds $m$ by more than a fixed number vanish on $V_m$, and each remaining
block is bounded by the same convolution estimate. The displayed estimate
carries no factor of the Riesz bounds, and it depends on the wavelet system
only through the radius of a ball containing the spectral support of the
scaling functions. The geometric summation that closes the argument is
secured by $s>0$, and the only trace of the degenerate limit $s \to 0^+$ is
carried in an explicit factor rather than concealed.

\medskip
\noindent\textbf{The identification of the linear approximation scale.}
Both inequalities of Theorem~\ref{thm:B3} are proved on the Littlewood--Paley
blocks of $f$, the inverse one through the vanishing of the high-frequency
blocks on $V_m$, and neither passes through
Lemma~\ref{lem:Bern:linear-jackson} or Theorem~\ref{thm:B2}, which record the
two directions at a single level $m$. Both hold at every $s>0$; the compact
spectral supports of the generators and of the scaling functions impose no
upper limit on the smoothness index. The summation device of the direct
inequality is the discrete Hardy inequality, which acts on the scalar block
sequence alone and sees no Muckenhoupt weight, so it transfers from the
unweighted statement without introducing a further power of $[w]_{A_p}$; in
the inverse inequality the low-frequency blocks are controlled by a
geometric series and each remaining block is matched with one level of the
approximation scale, so no weight enters the summation there either. The
weight powers are therefore carried entirely by the single-scale bounds:
the power $[w]_{A_p}^{2/p}$ of the direct inequality is the product of the
factors $[w]_{A_p}^{1/p}$ of the analysis bound
\eqref{eq:sse-analysis-general} and of the synthesis bound in
\eqref{eq:sse-refined}, and the power $[w]_{A_p}^{1/p}$ of the inverse
inequality is that of the convolution estimate for the Littlewood--Paley
blocks. The linear Jackson estimate of Lemma~\ref{lem:Bern:linear-jackson}
concerns $V_m$ on all of $\mathbb{R}^n$ and is not derived from
Theorem~\ref{thm:main1}, although both rest on the bounds
\eqref{eq:sse-refined} and \eqref{eq:coefficient-bound} and carry the same
power. When the function class is narrowed to the intersection with $L^p_w$,
the qualification propagates into the error definition and into the set
identity, and the identification is stated as an equality of subspaces of
$L^p_w$ rather than of abstract smoothness classes. The identification is
the weighted linear counterpart of the approximation-space characterizations
obtained for $N$-term wavelet approximation, where smoothness conditions
describe the nonlinear approximation spaces \cite{DominguezSeeger2023}.

\medskip
\noindent\textbf{Two special-case checks.} For $w\equiv 1$ one has
$[w]_{A_p}=1$, and \eqref{eq:intro-jackson-bound} reduces to the unweighted bound
$E^{\mathcal{F}}_N(f)_{L^p}\le C_0N^{-s/n}\|f\|_{\dot{B}^s_{p,q}}$ on the localized class with $w\equiv 1$, and the approximant that yields this bound in the proof of Theorem~\ref{thm:main1} is the truncation in scale along the dyadic tree of $Q_0$, the coefficient-level counterpart of linear approximation on a cube. For $p=2$ one has $2/p=2\theta(2)=1$, and the
dependence on the Muckenhoupt characteristic is linear in $[w]_{A_2}$. Both
checks are read directly from \eqref{eq:intro-jackson-bound}.

\medskip
\noindent\textbf{Relation to prior work.} The standing
hypothesis on the wavelet system follows the biorthogonal framework of
\cite{Gun2015}, and the present estimates refine the Jackson-type analysis
of \cite{Wang2021} in the weighted setting. For linear projections of
compactly supported biorthogonal multiresolution analyses on the real line,
explicit direct and inverse estimates on the inhomogeneous weighted Besov
scale, for smoothness indices below the number of vanishing moments and the
regularity of the wavelets, are obtained in \cite{Wang2026Wavelet}, where the
dependence on the weight is routed through the norm of the
Hardy--Littlewood maximal operator on $L^p_w$. The present estimates concern
band-limited systems on $\mathbb{R}^n$ and the homogeneous scale, admit
every $s>0$, display the power of $[w]_{A_p}$ itself, and identify the
approximation space. For $w\equiv 1$, the characterization of linear
approximation from dyadic shift-invariant spaces by Besov smoothness is
classical; in \cite[(4.27)]{Ronal1992} it follows from a pair of companion
inequalities through discrete Hardy inequalities, for smoothness orders
below both the smoothness and the order of polynomial reproduction of the
generator. For localized shallow networks on Besov classes
over bounded Lipschitz domains of $\mathbb{R}^n$, a Jackson-type estimate
with the rate $N^{-s/n}$ in the width $N$ and an explicit constant, tracked
through the activation and the domain, is established in
\cite{ChenWang2026Shallow}.

\medskip
\noindent\textbf{Organization of the paper.} Section~\ref{sec2} fixes the
normalization $\psi_{j,k}(x)=2^{jn/2}\psi(2^jx-k)$ and the dyadic cubes
$Q_{j,k}=2^{-j}(k+[0,1)^n)$ through which all weighted quantities are
expressed. It also records the standing hypothesis on the wavelet system,
the multiresolution condition used in Section~\ref{sec4}, and the weighted
objects of the analysis. It closes with two inequalities
for $A_p$ weights on cubes (Proposition~\ref{prop:ap-cube}). Section~\ref{sec3} proves the
single-scale estimate and the coefficient bound, introduces the localized
class, and establishes the Jackson-type inequality together with the lower
estimate. Section~\ref{sec4} proves a weighted bound for integral operators
whose kernels are concentrated at a single dyadic scale, applies it to the
Littlewood--Paley operators and the multiresolution projections, and
establishes in turn the single-scale Bernstein estimate, the Bernstein
inequality on $V_m$, the linear Jackson estimate and the identification of
the approximation scale, closing with a remark that compares this linear
identification with its nonlinear counterpart.

\section{Preliminaries}\label{sec2}

We begin by recalling the definitions and background material concerning
biorthogonal Riesz wavelet bases, weighted Besov spaces, Muckenhoupt $A_p$
weights, and the associated Littlewood--Paley decomposition. Familiarity
with the classical theory of wavelet bases and multiresolution analysis,
the theory of Besov and Triebel--Lizorkin spaces, and harmonic analysis in
the presence of Muckenhoupt weights is presupposed. Standard monographs for
the classical Fourier and weighted theory include
\cite{Grafakos2014Classical, Grafakos2009Modern}. For the construction and
structural properties of wavelet bases and frames we refer to
\cite{Daubechies1992, Meyer1992, Christensen2016}, for the theory of
function spaces to \cite{Triebel1992}, and for nonlinear $N$-term
approximation to \cite{DeVore1993, Temlyakov2011}.

Throughout this paper we work in the Euclidean space $\mathbb{R}^n$ with
dimension $n \ge 1$. The relation $X \lesssim Y$ indicates that
$X \le C Y$ for a positive constant $C$ whose dependence is either
irrelevant in the context or stated explicitly, and $X \sim Y$ means that
$X \lesssim Y$ and $Y\lesssim X$.

The Fourier transform of $f$ and its inverse are normalized as
\begin{equation*}
	\hat{f}(\xi):= \int_{\mathbb{R}^n} f(x)\, e^{-2\pi i x \cdot \xi}\, dx,
	\qquad
	f(x) = \int_{\mathbb{R}^n} \hat{f}(\xi)\, e^{2\pi i x \cdot \xi}\, d\xi.
\end{equation*}
We denote by $\mathcal{S}(\mathbb{R}^n)$ the Schwartz space, by
$\mathcal{S}'(\mathbb{R}^n)$ the space of tempered distributions, and by
$\mathcal{P}$ the space of polynomials.

We fix once and for all the wavelet normalization convention used in this
paper. Given a generator $\psi$, the dilated and translated system is
defined, for $j \in \mathbb{Z}$ and $k \in \mathbb{Z}^n$, by
\begin{equation}\label{eq:wavelet-normalization}
	\psi_{j,k}(x) := 2^{jn/2}\, \psi\!\left(2^j x - k\right),
\end{equation}
so that $\psi_{j,k}$ is spatially concentrated about the point $2^{-j}k$ and
satisfies $\|\psi_{j,k}\|_{L^2} = \|\psi\|_{L^2}$. The dual system is
normalized identically, and the same rule is applied to every generator of
Assumption~\ref{ass:wavelet-system}.

For $j \in \mathbb{Z}$ and $k \in \mathbb{Z}^n$ we denote by
\begin{equation}\label{eq:dyadic-cube}
	Q_{j,k}:= 2^{-j}\left(k+ [0,1)^n\right)
\end{equation}
the dyadic cube of generation $j$ with lower corner $2^{-j}k$ and Lebesgue
measure $|Q_{j,k}| =2^{-jn}$. For each fixed $j$ the family
$\{Q_{j,k}\}_{k\in \mathbb{Z}^n}$ tiles $\mathbb{R}^n$. For a weight $w$ and a measurable set $E$ we write
$w(E) :=\int_E w(x)\, dx$. All weighted quantities in the subsequent
estimates are expressed through the cube measure $w(Q_{j,k})$ rather than
through pointwise values of $w$.

For a measurable function $g$ and an exponent $0<p <\infty$, the weighted
$L^p$ quasi-norm associated with a weight $w$ is
\begin{equation}\label{eq:weighted-lp}
	\|g\|_{L^p_w}:= \left( \int_{\mathbb{R}^n}|g(x)|^p\, w(x)\, dx \right)^{1/p}.
\end{equation}
For $1 < p <\infty$ we write $p'$ for the conjugate exponent,
$1/p + 1/p'= 1$.

Given the system $\mathcal{F}= \{\psi_{j,k}\}$ of
Assumption~\ref{ass:wavelet-system} and an integer $N \ge 1$, the $N$-term
approximation error of $f$ in $L^p_w$ is
\begin{equation}\label{eq:nterm-error}
	E_N^{\mathcal{F}}(f)_{L^p_w}
	:= \inf\left\{\left\|f - \sum_{(j,k)\in\Lambda} d_{j,k}\, \psi_{j,k}
	\right\|_{L^p_w} :\#\Lambda \le N,\ \ d_{j,k}\in\mathbb{C} \right\},
\end{equation}
the infimum being taken over all index sets $\Lambda$ of cardinality at
most $N$ and all coefficient choices, where each index $(j,k)$ carries the
generator label of Assumption~\ref{ass:wavelet-system}.

The remainder of this section assembles the formal objects on which the
analysis of Sections~\ref{sec3} and~\ref{sec4} is based. We first record the standing
hypothesis on the wavelet system together with a concrete system satisfying
it, and then the multiresolution condition used in Section~\ref{sec4} together
with its verification for the same system. We then collect the definitions of the biorthogonal Riesz structure,
the weight class, the Littlewood--Paley apparatus, the weighted Besov
scale, and the decay classes of the generators. Two inequalities for $A_p$
weights on cubes close the section.

\begin{assumption}[Standing hypothesis on the wavelet system]
	\label{ass:wavelet-system}
	Fix an integer $L \ge 1$, generators $\psi^{(1)},\dots,\psi^{(L)}$, and
	dual generators $\tilde{\psi}^{(1)},\dots,\tilde{\psi}^{(L)}$. The systems
	$\{\psi^{(\epsilon)}_{j,k}\}$ and $\{\tilde{\psi}^{(\epsilon)}_{j,k}\}$,
	indexed by $1\le\epsilon\le L$, $j\in\mathbb{Z}$, $k\in\mathbb{Z}^n$ and
	normalized as in \eqref{eq:wavelet-normalization}, form a pair of
	\textit{band-limited} biorthogonal Riesz bases of $L^2(\mathbb{R}^n)$ with
	Riesz bounds $A, B$ as in Definition~\ref{def:biorthogonal-riesz}. The term
	band-limited means that every $\hat{\psi}^{(\epsilon)}$ and every
	$\hat{\tilde{\psi}}^{(\epsilon)}$ belongs to $C_c^\infty(\mathbb{R}^n)$ and
	is supported in the annulus $\{\xi : a \le |\xi| \le b\}$ for fixed
	$0<a<b$. In particular every generator and every dual generator belongs to
	$\mathcal{S}(\mathbb{R}^n)$ and has all moments vanishing. A statement that
	refers to a single generator $\psi$ with dual generator $\tilde{\psi}$,
	such as Lemma~\ref{lem:single-scale-estimate}, applies to each pair
	$(\psi^{(\epsilon)},\tilde{\psi}^{(\epsilon)})$ separately. When no
	confusion arises the label $\epsilon$ is suppressed, and $\{\psi_{j,k}\}$
	denotes the whole system. This hypothesis is declared here once and is
	assumed in every definition, theorem, and lemma that follows. It is not
	restated at each occurrence.
\end{assumption}

\begin{example}[Meyer systems]
	\label{ex:meyer-systems}
	For $n=1$ let $\psi$ be the Meyer wavelet and $\tilde{\psi}=\psi$. Its
	Fourier transform belongs to $C_c^\infty(\mathbb{R})$ and, in the
	normalization of the Fourier transform fixed above, is supported in
	$\{1/3\le|\xi|\le 4/3\}$, while the associated scaling function
	$\varphi_{\mathrm M}$ has $\hat{\varphi}_{\mathrm M}$ supported in
	$\{|\xi|\le 2/3\}$ \cite[(4.2.3) and Section~5.2, pp.~137--139]{Daubechies1992}, after conversion from the
	angular-frequency normalization used there and with the auxiliary function
	$\nu$ of that construction chosen in $C^\infty(\mathbb{R})$. The system
	$\{\psi_{j,k}\}$ is an orthonormal basis of $L^2(\mathbb{R})$, so
	Assumption~\ref{ass:wavelet-system} holds with $L=1$ and $A=B=1$. For
	$n\ge 2$ put $\vartheta_0:=\varphi_{\mathrm M}$ and $\vartheta_1:=\psi$,
	and for $\epsilon\in\{0,1\}^n\setminus\{0\}$ let
	$\psi^{(\epsilon)}(x):=\prod_{i=1}^n\vartheta_{\epsilon_i}(x_i)$. These
	$2^n-1$ tensor-product functions generate an orthonormal basis of
	$L^2(\mathbb{R}^n)$ through \eqref{eq:wavelet-normalization}
	\cite[Section~10.1]{Daubechies1992}. Each $\hat{\psi}^{(\epsilon)}$ is a
	product of functions in $C_c^\infty(\mathbb{R})$. At least one factor
	equals $\hat{\psi}$, which vanishes for $|\xi_i|<1/3$, so
	$\hat{\psi}^{(\epsilon)}$ is supported in
	$\{1/3\le|\xi|\le 4\sqrt{n}/3\}$. Hence, after an enumeration of
	$\{0,1\}^n\setminus\{0\}$ by $1,\dots,2^n-1$,
	Assumption~\ref{ass:wavelet-system} holds with $L=2^n-1$,
	$\tilde{\psi}^{(\epsilon)}=\psi^{(\epsilon)}$, and $A=B=1$.
\end{example}

\begin{assumption}[Multiresolution condition]
	\label{ass:mra}
	There are functions $\chi, \tilde{\chi}$ with
	$\hat{\chi}, \hat{\tilde{\chi}} \in C_c^\infty(\mathbb{R}^n)$ supported in
	$\{\xi : |\xi| \le R\}$ for some $R > 0$, such that
	$\langle \chi_{0,k}, \tilde{\chi}_{0,k'} \rangle = \delta_{k,k'}$ for
	$k, k' \in \mathbb{Z}^n$ and, for every $m \in \mathbb{Z}$ and every
	$f \in L^2(\mathbb{R}^n)$,
	\begin{equation}\label{eq:b-mra-identity}\tag{M}
		\sum_{\epsilon=1}^{L} \sum_{j < m} \sum_{k \in \mathbb{Z}^n}
		\langle f, \tilde{\psi}^{(\epsilon)}_{j,k} \rangle\,
		\psi^{(\epsilon)}_{j,k}
		= \sum_{k \in \mathbb{Z}^n}
		\langle f, \tilde{\chi}_{m,k} \rangle\, \chi_{m,k},
	\end{equation}
	both series converging unconditionally in $L^2(\mathbb{R}^n)$. Here
	$\chi_{m,k}$ and $\tilde{\chi}_{m,k}$ are formed as in
	\eqref{eq:wavelet-normalization}. This hypothesis is imposed in addition
	to Assumption~\ref{ass:wavelet-system} and is used only in
	Section~\ref{sec4}. A pair $\chi, \tilde{\chi}$ and a radius $R$ as above
	are fixed, and every object formed from them, including the spaces $V_m$
	of Definition~\ref{def:b-multiresolution}, refers to this choice.
\end{assumption}

Assumption~\ref{ass:mra} is not derived from
Assumption~\ref{ass:wavelet-system}, and no passage from a band-limited
wavelet system to a multiresolution structure is asserted. The Meyer systems
of Example~\ref{ex:meyer-systems} satisfy both hypotheses.

\begin{example}[Meyer systems and the multiresolution condition]
	\label{ex:meyer-mra}
	Let the system of Assumption~\ref{ass:wavelet-system} be that of
	Example~\ref{ex:meyer-systems}, let $\varphi_{\mathrm M}$ and $\psi$ be
	the Meyer scaling function and the Meyer wavelet used there, and put
	$\vartheta_0 := \varphi_{\mathrm M}$ and $\vartheta_1 := \psi$ also for
	$n = 1$. Put
	$\chi(x) = \tilde{\chi}(x) := \prod_{i=1}^n \varphi_{\mathrm M}(x_i)$ and
	$R := 2\sqrt{n}/3$. Then Assumption~\ref{ass:mra} holds. Indeed,
	$\hat{\varphi}_{\mathrm M} \in C_c^\infty(\mathbb{R})$ is supported in
	$\{|\xi| \le 2/3\}$ by Example~\ref{ex:meyer-systems}, so that
	$\hat{\chi}(\xi) = \prod_i \hat{\varphi}_{\mathrm M}(\xi_i)$ belongs to
	$C_c^\infty(\mathbb{R}^n)$ and is supported in
	$\{|\xi|_\infty \le 2/3\} \subset \{|\xi| \le R\}$. The function
	$\varphi_{\mathrm M}$ is the scaling function of an orthonormal
	multiresolution analysis of $L^2(\mathbb{R})$ with wavelet $\psi$, and the
	products $\prod_{i=1}^n \vartheta_{\epsilon_i}(x_i)$,
	$\epsilon \in \{0,1\}^n$, form the associated tensor-product
	multiresolution analysis of $L^2(\mathbb{R}^n)$ with scaling function
	$\chi$ \cite[Sections~5.1, 5.2 and~10.1]{Daubechies1992}. The
	tensor-product construction is written out in Section~10.1 of that
	reference for $n=2$, and its extension to $n$ factors is direct. In the
	dilation convention of \eqref{eq:wavelet-normalization}, which is that of
	\cite{Daubechies1992} after the relabeling $j \mapsto -j$, this means
	that for every $m \in \mathbb{Z}$ both $\{\chi_{m,k}\}_{k}$ and
	$\{\psi^{(\epsilon)}_{j,k}\}_{\epsilon,\, j < m,\, k}$ are orthonormal
	bases of one closed subspace of $L^2(\mathbb{R}^n)$; in particular
	$\langle \chi_{0,k}, \chi_{0,k'} \rangle = \delta_{k,k'}$, and
	\eqref{eq:b-mra-identity} expresses the orthogonal projection onto this
	subspace in the two bases, both expansions converging unconditionally.
\end{example}

\begin{definition}[{Biorthogonal Riesz wavelet basis and dual system, see, e.g.,~\cite[Section~2]{Gun2015}}]
	\label{def:biorthogonal-riesz}
	A system $\{\psi^{(\epsilon)}_{j,k}\}$ in $L^2(\mathbb{R}^n)$, indexed by
	$1\le\epsilon\le L$, $j\in\mathbb{Z}$, $k\in\mathbb{Z}^n$, is a Riesz
	basis if its closed linear span is $L^2(\mathbb{R}^n)$ and there exist
	constants $A, B > 0$, the Riesz bounds, such that
	\begin{equation}\label{eq:riesz-bounds}
		A \sum_{\epsilon,j,k} |c^{(\epsilon)}_{j,k}|^2
		\le  \Bigl\| \sum_{\epsilon,j,k} c^{(\epsilon)}_{j,k}\,
		\psi^{(\epsilon)}_{j,k} \Bigr\|_{L^2}^2
		\le  B \sum_{\epsilon,j,k} |c^{(\epsilon)}_{j,k}|^2
	\end{equation}
	for every finitely supported coefficient sequence. Two systems
	$\{\psi^{(\epsilon)}_{j,k}\}$ and $\{\tilde{\psi}^{(\epsilon)}_{j,k}\}$
	form a pair of biorthogonal Riesz wavelet bases if each is a Riesz basis
	of the form \eqref{eq:wavelet-normalization}, the biorthogonality relation
	\begin{equation}\label{eq:biorthogonality}
		\langle \psi^{(\epsilon)}_{j,k}, \tilde{\psi}^{(\epsilon')}_{j',k'} \rangle
		= \delta_{\epsilon,\epsilon'}\,\delta_{j,j'}\, \delta_{k,k'}
	\end{equation}
	holds, and every $f \in L^2(\mathbb{R}^n)$ admits the expansion
	\begin{equation}\label{eq:dual-expansion}
		f = \sum_{\epsilon=1}^{L}\sum_{j\in\mathbb{Z}}\sum_{k\in\mathbb{Z}^n}
		\langle f, \tilde{\psi}^{(\epsilon)}_{j,k} \rangle\, \psi^{(\epsilon)}_{j,k}
	\end{equation}
	with convergence in $L^2(\mathbb{R}^n)$. The dual system is generated by
	the finitely many functions $\tilde{\psi}^{(\epsilon)}$ through
	\eqref{eq:wavelet-normalization}. This wavelet structure of the dual
	system is not automatic for general Riesz bases, as the counterexample in
	\cite[Remark~2.2(3)]{Gun2015} shows. It is therefore adopted as part of the
	standing hypothesis, and conditions guaranteeing it are given in
	\cite[Theorem~3.4]{Gun2015}.
\end{definition}

\begin{definition}[{Muckenhoupt class $A_p$ and characteristic $[w]_{A_p}$ \cite[Definition 9.1.3]{Grafakos2009Modern}}]\label{def:muckenhoupt}
	Let $1 < p < \infty$. A weight $w$, that is, a locally integrable and
	almost everywhere positive function on $\mathbb{R}^n$, belongs to the
	Muckenhoupt class $A_p$ if its $A_p$ characteristic
	\begin{equation}\label{eq:ap-characteristic}
		[w]_{A_p} := \sup_{Q}
		\left( \frac{1}{|Q|} \int_Q w(x)\, dx \right)
		\left( \frac{1}{|Q|} \int_Q w(x)^{-\frac{1}{p-1}}\, dx \right)^{p-1}
	\end{equation}
	is finite, where the supremum is taken over all cubes
	$Q \subset \mathbb{R}^n$ with sides parallel to the coordinate axes. By
	H\"older's inequality $[w]_{A_p}\ge 1$ for every $w\in A_p$.
\end{definition}

\begin{definition}[{Littlewood--Paley operators, see, e.g., \cite[Section~6.2.2]{Grafakos2009Modern}}]\label{def:littlewood-paley}
	Fix $\varphi\in\mathcal{S}(\mathbb{R}^n)$ such that $\hat{\varphi}$ is
	real-valued and radial, is supported in the annulus
	$\{\xi : c_1 \le |\xi| \le c_2\}$ for fixed $0<c_1<c_2$, and satisfies
	$\sum_{j\in\mathbb{Z}}\hat{\varphi}(2^{-j}\xi)=1$ for every $\xi\ne 0$.
	For $j\in\mathbb{Z}$ put $\varphi_j(x):=2^{jn}\varphi(2^jx)$, so that
	$\hat{\varphi}_j(\xi)=\hat{\varphi}(2^{-j}\xi)$. The homogeneous
	Littlewood--Paley operators are
	\begin{equation}\label{eq:lp-multiplier}
		\Delta_j f := \varphi_j * f,
		\qquad
		\widehat{\Delta_j f} = \hat{\varphi}_j\, \hat{f},
		\qquad j\in\mathbb{Z}.
	\end{equation}
	Since $\hat{\varphi}$ is real-valued and radial, $\varphi$ is real-valued
	and even. Since $\hat{\varphi}$ vanishes near the origin, each $\Delta_j$
	annihilates polynomials and is well defined on
	$\mathcal{S}'(\mathbb{R}^n)/\mathcal{P}$.
\end{definition}

\begin{definition}[Homogeneous weighted Besov space]\label{def:weighted-besov}
	Let $s > 0$, $1 < p<\infty$, $0 < q < \infty$, and $w \in A_p$. The
	homogeneous weighted Besov space $\dot{B}^s_{p,q}(w)$ consists of all
	$f \in \mathcal{S}'(\mathbb{R}^n)/\mathcal{P}$ for which the quasi-norm
	\begin{equation}\label{eq:weighted-besov-norm}
		\|f\|_{\dot{B}^s_{p,q}(w)} := \left( \sum_{j\in\mathbb{Z}} 2^{jsq}\,
		\|\Delta_j f\|_{L^p_w}^q \right)^{1/q}
	\end{equation}
	is finite. A single term of \eqref{eq:weighted-besov-norm} is dominated by
	the whole sum, so $2^{js}\|\Delta_j f\|_{L^p_w}\le\|f\|_{\dot{B}^s_{p,q}(w)}$
	for every $j\in\mathbb{Z}$ and every $0<q<\infty$. The homogeneous
	weighted spaces of \cite[p.~220, (1.1)]{BuiPaluszynskiTaibleson1996} are
	defined through Littlewood--Paley blocks of the same kind for
	$w\in A_\infty$, and anisotropic spaces of this type are treated in the
	doubling-measure framework of \cite[Section~3]{Bownik2005}.
\end{definition}

We record the finite-band structure that links the Littlewood--Paley blocks
with the frequency support of the generators. Put
\begin{equation}\label{eq:finite-band}
	M:=\Bigl\lceil\max\Bigl\{\log_2\frac{b}{c_1},\ \log_2\frac{c_2}{a},\ 0\Bigr\}\Bigr\rceil,
	\qquad
	C_M:=\sum_{|i|\le M}2^{is},
	\qquad
	C'_M:=\left(\sum_{|i|\le M}2^{isq}\right)^{1/q}.
\end{equation}
Let $h\in\mathcal{S}(\mathbb{R}^n)$ with $\operatorname{supp}\hat{h}$
contained in $\{2^j a\le|\xi|\le 2^j b\}$ for some $j\in\mathbb{Z}$. If
$\hat{\varphi}_{j'}\hat{h}\not\equiv 0$, then some $\xi$ satisfies both
$c_12^{j'}\le|\xi|\le c_22^{j'}$ and $2^ja\le|\xi|\le 2^jb$, so that
$j'-j\le\log_2(b/c_1)$ and $j-j'\le\log_2(c_2/a)$. Consequently
$\Delta_{j'}h=0$ for $|j'-j|>M$. Since the dilates of $\hat{\varphi}$ sum to
one away from the origin, $h=\sum_{|i|\le M}\Delta_{j+i}h$. Because the
range $|i|\le M$ is symmetric, $C_M=\sum_{|i|\le M}2^{-is}$ as well.

\begin{definition}[Decay classes]\label{def:decay-classes}
	Let $\tau > 0$. A function $\psi$ belongs to the synthesis decay class
	$M^\tau$ if
	\begin{equation}\label{eq:decay-mtau}
		|\psi|_{M^\tau}:=\sup_{x\in\mathbb{R}^n}(1+|x|)^{2n+\tau}\,|\psi(x)| < \infty.
	\end{equation}
	Let $\sigma > 2n$. A function $\tilde{\psi}$ belongs to the analysis decay
	class $\mathcal{D}_\sigma$ if
	\begin{equation}\label{eq:decay-dsigma}
		|\tilde{\psi}|_{\mathcal{D}_\sigma}:=\int_{\mathbb{R}^n} \left(1+|x|^\sigma\right) |\tilde{\psi}(x)| \, dx < \infty.
	\end{equation}
	Under Assumption~\ref{ass:wavelet-system} every generator belongs to
	$M^\tau$ for every $\tau>0$, and every dual generator belongs to
	$\mathcal{D}_\sigma$ for every $\sigma>2n$. The role of the two classes is
	to name the quantities through which the constants depend on the
	generators. A constant said to depend on the decay parameter $\tau$
	depends on $\tau$ and on $|\psi|_{M^\tau}$. A constant said to depend on
	the decay parameter $\sigma$ depends on $\sigma$, on
	$|\tilde{\psi}|_{\mathcal{D}_\sigma}$, and on the annulus radii $a,b$ of
	Assumption~\ref{ass:wavelet-system}. In both cases the maximum over the
	generator label $\epsilon$ is taken. For $\tilde{\psi}\in\mathcal{D}_\sigma$
	as in Assumption~\ref{ass:wavelet-system} the pointwise quantity
	\begin{equation}\label{eq:decay-pointwise}
		\mathrm{L}_\sigma(\tilde{\psi}):=\sup_{x\in\mathbb{R}^n}(1+|x|)^{\sigma}\,|\tilde{\psi}(x)|
		\le 2^{\sigma}K_\Phi\,|\tilde{\psi}|_{\mathcal{D}_\sigma}
	\end{equation}
	is finite. Here $\Phi\in\mathcal{S}(\mathbb{R}^n)$ is a fixed function, chosen depending only on $n$, $a$ and $b$, with
	$\hat{\Phi}=1$ on $\{a\le|\xi|\le b\}$, and
	$K_\Phi:=\sup_{z}(1+|z|)^\sigma|\Phi(z)|$. Indeed
	$\tilde{\psi}=\Phi*\tilde{\psi}$, while $(1+|x|)^\sigma\le(1+|x-y|)^\sigma(1+|y|)^\sigma$
	and $(1+|y|)^\sigma\le 2^\sigma(1+|y|^\sigma)$.
\end{definition}

The following two inequalities are the only consequences of the $A_p$
condition used in this paper.

\begin{proposition}[Two inequalities for $A_p$ weights on cubes]
	\label{prop:ap-cube}
	Let $1<p<\infty$, $w\in A_p$, and let $Q\subset\mathbb{R}^n$ be a cube with sides parallel to the
	coordinate axes.
	For every measurable function $g$ on $Q$,
	\begin{equation}\label{eq:ap-holder}
		\int_Q |g(x)|\,dx
		\le [w]_{A_p}^{1/p}\,|Q|\,w(Q)^{-1/p}
		\left(\int_Q |g(x)|^p\,w(x)\,dx\right)^{1/p},
	\end{equation}
	and for every measurable set $S\subset Q$ with $|S|>0$,
	\begin{equation}\label{eq:ap-doubling}
		w(Q)\le [w]_{A_p}\left(\frac{|Q|}{|S|}\right)^{p} w(S).
	\end{equation}
\end{proposition}

\begin{proof}[Proof of Proposition~\ref{prop:ap-cube}.]
	We write $|g| = |g|\,w^{1/p}\,w^{-1/p}$ and apply H\"older's inequality on
	$Q$ with the exponents $p$ and $p'$. Since $p'/p = 1/(p-1)$ and
	$1/p'=(p-1)/p$,
	\[
	\int_Q|g| \le \left(\int_Q|g|^p w\right)^{1/p}
	\left(\int_Q w^{-\frac{1}{p-1}}\right)^{1/p'}
	= \left(\int_Q|g|^p w\right)^{1/p}|Q|^{1/p'}
	\left[\left(\frac{1}{|Q|}\int_Q w^{-\frac{1}{p-1}}\right)^{p-1}\right]^{1/p}.
	\]
	By \eqref{eq:ap-characteristic} the bracket does not exceed
	$[w]_{A_p}\,|Q|\,w(Q)^{-1}$. Since $|Q|^{1/p'}|Q|^{1/p}=|Q|$, this gives
	\eqref{eq:ap-holder}. Taking $g=\mathbf{1}_S$ in \eqref{eq:ap-holder}
	yields $|S|\le[w]_{A_p}^{1/p}|Q|\,w(Q)^{-1/p}w(S)^{1/p}$, and raising both
	sides to the power $p$ gives \eqref{eq:ap-doubling}.
\end{proof}

Inequality \eqref{eq:ap-doubling} is a quantitative form of the doubling
property of $A_p$ weights. The proof above is included because the constants of Section~\ref{sec3} are read off
from \eqref{eq:ap-holder} and \eqref{eq:ap-doubling} directly.

\section{The Quantitative Machinery: Single-Scale Estimates, Jackson-type Bounds, and Extremal Families}\label{sec3}

This section proves the single-scale estimate and a coefficient bound,
introduces the localized class, and establishes the Jackson-type inequality
together with the lower estimate. Throughout, $L$, $a$, $b$ are as in
Assumption~\ref{ass:wavelet-system}, and $M$, $C_M$, $C'_M$ are as in
\eqref{eq:finite-band}.

\begin{lemma}[Single-scale remainder norm estimate]
	\label{lem:single-scale-estimate}
	Let Assumption~\ref{ass:wavelet-system} hold, with synthesis generator
	$\psi \in M^\tau$ and dual generator $\tilde{\psi} \in \mathcal{D}_\sigma$
	for some $\tau > 0$ and $\sigma > 2n$
	\textnormal{(Definition~\ref{def:decay-classes})}. Fix $j \in \mathbb{Z}$,
	a finite set $\Lambda \subset \mathbb{Z}^n$, and coefficients
	$\{c_{j,k}\}_{k\in\Lambda} \subset \mathbb{C}$; set
	$r_\Lambda^{(j)}:= \sum\limits_{k\in\Lambda} c_{j,k}\,\psi_{j,k}$. Then, for
	every $1 < p< \infty$ and every $w \in A_p$,
	
	\begin{align}\label{eq:single-scale-lemma}
			c\,[w]_{A_p}^{-\theta(p)}\, 2^{jn/2}
		\left(\sum_{k\in\Lambda} |c_{j,k}|^p\,w(Q_{j,k})\right)^{1/p}
		 &\le 
		\big\|r_\Lambda^{(j)} \big\|_{L^p_w}\nonumber\\
		& \le 
		C\,[w]_{A_p}^{\theta(p)}\, 2^{jn/2}
		\left(\sum_{k\in\Lambda} |c_{j,k}|^p\, w(Q_{j,k})\right)^{1/p}
	\end{align}
	where
	\begin{equation}\label{eq:single-scale-theta}
		\theta(p)=  \tfrac{1}{2}\,\max\!\left\{1,\ (p-1)^{-1}\right\}.
	\end{equation}
	The constant $C>0$ depends only on $n$, $p$, and the decay parameter $\tau$, and the constant $c>0$ depends only on $n$, $p$, and the decay parameter $\sigma$. Both constants are independent of $j$, $\Lambda$, and $\{c_{j,k}\}$.
\end{lemma}

\begin{proof}[Proof of Lemma \ref{lem:single-scale-estimate}.]
	The argument has three steps. The first step reduces the estimate to the
	scale $j=0$ by dilation. The second step proves the upper estimate with
	the power $[w]_{A_p}^{1/p}$ by a cube-by-cube synthesis argument. The
	third step proves, for every $g\in L^p_w$, a weighted bound for the
	analysis coefficients of $g$ with the same power, and reads off the lower
	estimate from biorthogonality. Since $[w]_{A_p}\ge 1$ and
	$1/p\le\theta(p)$ for every $1<p<\infty$, the two estimates imply
	\eqref{eq:single-scale-lemma}. The inequality $1/p\le\theta(p)$ holds
	because $1/p\le 1/2$ for $p\ge 2$, while for $1<p<2$ the inequality
	$1/p\le 1/(2(p-1))$ is equivalent to $p\le 2$.
	
	\emph{Step 1: reduction to $j=0$.} For $w\in A_p$ and $j\in\mathbb{Z}$ we
	put $w_j(y):=w(2^{-j}y)$. The dilation $x\mapsto 2^jx$ maps cubes onto
	cubes and preserves the averages in \eqref{eq:ap-characteristic}, so
	$[w_j]_{A_p}=[w]_{A_p}$. The change of variables $y=2^jx$ gives
	\begin{equation}\label{eq:sse-rescale}
		\big\|r^{(j)}_\Lambda\big\|_{L^p_w}
		= 2^{jn(\frac12-\frac1p)}
		\Bigl\|\sum_{k\in\Lambda}c_{j,k}\,\psi(\cdot-k)\Bigr\|_{L^p_{w_j}},
		\qquad
		w(Q_{j,k}) = 2^{-jn}\,w_j(Q_{0,k}).
	\end{equation}
	Consequently
	$2^{jn/2}\bigl(\sum_{k}|c_{j,k}|^p\,w(Q_{j,k})\bigr)^{1/p}
	=2^{jn(\frac12-\frac1p)}\bigl(\sum_{k}|c_{j,k}|^p\,w_j(Q_{0,k})\bigr)^{1/p}$,
	so the two sides of \eqref{eq:single-scale-lemma} carry the common factor
	$2^{jn(\frac12-\frac1p)}$. Since $w_j$ ranges over $A_p$ with the same
	characteristic, it suffices to treat $j=0$ with constants that depend on
	the weight only through $[w]_{A_p}$. We write $Q_k:=Q_{0,k}$,
	$c_k:=c_{0,k}$, and $r:=\sum_{k\in\Lambda}c_k\,\psi(\cdot-k)$. For
	$m\ge 0$ we put $Q_k^{(m)}:=k+[-m,m+1)^n$, which is the union of the
	$(2m+1)^n$ cubes $Q_{k'}$ with $|k'-k|_\infty\le m$ and has
	$|Q_k^{(m)}|=(2m+1)^n$. We also put $Q_k^{(-1)}:=\emptyset$. By
	\eqref{eq:ap-doubling} with $S=Q_k$ and $Q=Q_k^{(m)}$,
	\begin{equation}\label{eq:sse-growth}
		w\bigl(Q_k^{(m)}\bigr)\le[w]_{A_p}\,(2m+1)^{np}\,w(Q_k),
		\qquad m\ge 0.
	\end{equation}
	
	\emph{Step 2: the upper estimate.} We introduce the decay envelope of
	$\psi$, namely $\beta_0:=\sup|\psi|$ and
	$\beta_m:=\sup\{|\psi(z)| : |z|_\infty\ge m-1\}$ for $m\ge 1$. The sequence
	$(\beta_m)$ is nonincreasing and tends to zero. If $|z|_\infty\ge m-1\ge 1$
	then $1+|z|\ge m$, so \eqref{eq:decay-mtau} gives
	\begin{equation}\label{eq:sse-envelope}
		\beta_m\le|\psi|_{M^\tau}\,\max\{1,m\}^{-2n-\tau},
		\qquad m\ge 0.
	\end{equation}
	Fix $k_0\in\mathbb{Z}^n$ and $x\in Q_{k_0}$. If $|k-k_0|_\infty=m\ge 1$,
	then one coordinate of $k_0-k$ equals $\pm m$. The corresponding
	coordinate of $x-k$ lies in $[m,m+1)$ or in $[-m,-m+1)$, and hence
	$|x-k|_\infty\ge m-1$. Therefore $|\psi(x-k)|\le\beta_m$ whenever
	$|k-k_0|_\infty=m$, for every $m\ge 0$. We write
	$\beta_m=\sum_{i\ge m}(\beta_i-\beta_{i+1})$ and exchange the order of
	summation, which is legitimate because all terms are nonnegative. This
	gives, for $x\in Q_{k_0}$,
	\begin{equation}\label{eq:sse-pointwise}
		|r(x)|
		\le\sum_{m\ge 0}\beta_m\sum_{\substack{k\in\Lambda\\|k-k_0|_\infty=m}}|c_k|
		=\sum_{i\ge 0}(\beta_i-\beta_{i+1})\,A_i(k_0),
		\qquad
		A_i(k_0):=\sum_{\substack{k\in\Lambda\\|k-k_0|_\infty\le i}}|c_k|.
	\end{equation}
	The right-hand side of \eqref{eq:sse-pointwise} is constant on each cube
	$Q_{k_0}$. With $F_i:=\sum_{k_0}A_i(k_0)\mathbf{1}_{Q_{k_0}}$, Minkowski's
	inequality in $L^p_w$ yields
	$\|r\|_{L^p_w}\le\sum_{i\ge 0}(\beta_i-\beta_{i+1})\|F_i\|_{L^p_w}$. By
	H\"older's inequality on the $(2i+1)^n$ indices $k$ with
	$|k-k_0|_\infty\le i$ we have
	$A_i(k_0)^p\le(2i+1)^{n(p-1)}\sum_{|k-k_0|_\infty\le i}|c_k|^p$. Summing
	over $k_0$, interchanging the order of summation, and using
	$\sum_{|k_0-k|_\infty\le i}w(Q_{k_0})=w(Q_k^{(i)})$ together with
	\eqref{eq:sse-growth}, we obtain
		\begin{align*}
				\|F_i\|_{L^p_w}^p=\sum_{k_0}w(Q_{k_0})A_i(k_0)^p
			&\le(2i+1)^{n(p-1)}\sum_{k\in\Lambda}|c_k|^p\,w\bigl(Q_k^{(i)}\bigr)\\
			&\le[w]_{A_p}\,(2i+1)^{n(2p-1)}\sum_{k\in\Lambda}|c_k|^p\,w(Q_k).
	\end{align*}
	Since $n(2p-1)/p=n+n/p'$, the preceding two displays combine into
	\begin{equation}\label{eq:sse-upper-refined}
		\|r\|_{L^p_w}\le K_U\,[w]_{A_p}^{1/p}
		\Bigl(\sum_{k\in\Lambda}|c_k|^p\,w(Q_k)\Bigr)^{1/p},
		\qquad
		K_U:=\sum_{i\ge 0}(\beta_i-\beta_{i+1})\,(2i+1)^{n+n/p'}.
	\end{equation}
	The constant $K_U$ is finite and explicitly bounded. Put $\gamma:=n+n/p'$,
	$\alpha_i:=(2i+1)^\gamma$, and $\alpha_{-1}:=0$. Reversing the exchange of summation
	gives $K_U=\sum_{m\ge 0}\beta_m(\alpha_m-\alpha_{m-1})$. For $m\ge 1$ the mean value
	theorem gives $\alpha_m-\alpha_{m-1}\le 2\gamma(2m+1)^{\gamma-1}\le 2\gamma\,3^{\gamma-1}m^{\gamma-1}$,
	and since $\gamma-1-2n-\tau=-1-\tau-n/p$, the envelope bound
	\eqref{eq:sse-envelope} yields
	\begin{equation}\label{eq:sse-KU}
		K_U\le|\psi|_{M^\tau}\Bigl(1+2\gamma\,3^{\gamma-1}\,\zeta\bigl(1+\tau+n/p\bigr)\Bigr),
		\qquad \gamma=n+\frac{n}{p'},
	\end{equation}
	where $\zeta$ is the Riemann zeta function. The right-hand side of
	\eqref{eq:sse-KU} depends only on $n$, $p$, $\tau$ and $|\psi|_{M^\tau}$.
	Undoing the dilation of Step 1, we arrive at
	$\|r^{(j)}_\Lambda\|_{L^p_w}\le K_U[w]_{A_p}^{1/p}2^{jn/2}\bigl(\sum_{k\in\Lambda}|c_{j,k}|^pw(Q_{j,k})\bigr)^{1/p}$.
	Since $[w]_{A_p}^{1/p}\le[w]_{A_p}^{\theta(p)}$ and $K_U$ does not exceed
	the right-hand side of \eqref{eq:sse-KU}, this is the upper estimate in
	\eqref{eq:single-scale-lemma} with $C$ equal to the right-hand side of
	\eqref{eq:sse-KU}.
	
	\emph{Step 3: the analysis bound and the lower estimate.} Let
	$g\in L^p_w$. We introduce the decay envelope of $\tilde{\psi}$, namely
	$\tilde{\beta}_0:=\sup|\tilde{\psi}|$ and
	$\tilde{\beta}_m:=\sup\{|\tilde{\psi}(z)|:|z|_\infty\ge m-1\}$ for $m\ge 1$.
	By \eqref{eq:decay-pointwise} it satisfies
	$\tilde{\beta}_m\le\mathrm{L}_\sigma(\tilde{\psi})\max\{1,m\}^{-\sigma}$.
	For $m\ge 1$ and $y\in Q_k^{(m)}\setminus Q_k^{(m-1)}$, some coordinate of
	$y-k$ lies in $[-m,-m+1)\cup[m,m+1)$, so $|y-k|_\infty\ge m-1$ and
	$|\tilde{\psi}(y-k)|\le\tilde{\beta}_m$. The same exchange of summation as
	in \eqref{eq:sse-pointwise} therefore gives
	\[
	\int_{\mathbb{R}^n}|g(y)|\,|\tilde{\psi}(y-k)|\,dy
	\le\sum_{m\ge 0}\tilde{\beta}_m\int_{Q_k^{(m)}\setminus Q_k^{(m-1)}}|g|
	=\sum_{i\ge 0}(\tilde{\beta}_i-\tilde{\beta}_{i+1})\int_{Q_k^{(i)}}|g|.
	\]
	We apply \eqref{eq:ap-holder} with $Q=Q_k^{(i)}$ and use
	$w(Q_k)\le w(Q_k^{(i)})$. This gives
	$w(Q_k)^{1/p}\int_{Q_k^{(i)}}|g|\le[w]_{A_p}^{1/p}(2i+1)^n G_{k,i}$ with
	$G_{k,i}:=\left(\int_{Q_k^{(i)}}|g|^pw\right)^{1/p}$, and hence
	\begin{equation}\label{eq:sse-single-k}
		w(Q_k)^{1/p}\int_{\mathbb{R}^n}|g(y)|\,|\tilde{\psi}(y-k)|\,dy
		\le[w]_{A_p}^{1/p}\sum_{i\ge 0}(\tilde{\beta}_i-\tilde{\beta}_{i+1})(2i+1)^n\,G_{k,i}.
	\end{equation}
	We now take the $\ell^p$ norm in $k$ and apply Minkowski's inequality in
	the index $i$. Every point of $\mathbb{R}^n$ lies in exactly $(2i+1)^n$ of
	the cubes $Q_k^{(i)}$, so $\sum_kG_{k,i}^p=(2i+1)^n\|g\|_{L^p_w}^p$.
	Consequently
	\begin{equation}\label{eq:sse-analysis-zero}
		\left(\sum_{k\in\mathbb{Z}^n}|\langle g,\tilde{\psi}_{0,k}\rangle|^p\,w(Q_k)\right)^{1/p}
		\le K_L\,[w]_{A_p}^{1/p}\,\|g\|_{L^p_w},
		\qquad
		K_L:=\sum_{i\ge 0}(\tilde{\beta}_i-\tilde{\beta}_{i+1})(2i+1)^{n+n/p}.
	\end{equation}
	With $\gamma':=n+n/p$, the argument that led to \eqref{eq:sse-KU} gives
	\begin{equation}\label{eq:sse-KL}
		K_L\le\mathrm{L}_\sigma(\tilde{\psi})\Bigl(1+2\gamma'\,3^{\gamma'-1}\,\zeta\bigl(1+\sigma-n-n/p\bigr)\Bigr),
		\qquad \gamma'=n+\frac{n}{p},
	\end{equation}
	which is finite because $\sigma>2n>n+n/p$. The right-hand side of
	\eqref{eq:sse-KL} depends only on $n$, $p$, $\sigma$ and
	$\mathrm{L}_\sigma(\tilde{\psi})$. By \eqref{eq:decay-pointwise} it does not
	exceed the quantity obtained by replacing $\mathrm{L}_\sigma(\tilde{\psi})$
	with $2^{\sigma}K_\Phi|\tilde{\psi}|_{\mathcal{D}_\sigma}$, and this quantity
	depends only on $n$, $p$ and the decay parameter $\sigma$ in the sense of
	Definition~\ref{def:decay-classes}. For general $j$ we put
	$g_j(y):=g(2^{-j}y)$. Then
	$\langle g,\tilde{\psi}_{j,k}\rangle=2^{-jn/2}\langle g_j,\tilde{\psi}_{0,k}\rangle$
	and $\|g_j\|_{L^p_{w_j}}=2^{jn/p}\|g\|_{L^p_w}$, and together with
	$w(Q_{j,k})=2^{-jn}w_j(Q_{0,k})$ the bound \eqref{eq:sse-analysis-zero}
	becomes
	\begin{equation}\label{eq:sse-analysis-general}
		2^{jn/2}\left(\sum_{k\in\mathbb{Z}^n}|\langle g,\tilde{\psi}_{j,k}\rangle|^p\,w(Q_{j,k})\right)^{1/p}
		\le K_L\,[w]_{A_p}^{1/p}\,\|g\|_{L^p_w},
		\qquad g\in L^p_w,\ j\in\mathbb{Z}.
	\end{equation}
	We apply \eqref{eq:sse-analysis-general} to $g=r^{(j)}_\Lambda$. By the
	biorthogonality relation \eqref{eq:biorthogonality},
	$\langle r^{(j)}_\Lambda,\tilde{\psi}_{j,k}\rangle=c_{j,k}$ for
	$k\in\Lambda$. Discarding the indices $k\notin\Lambda$, we obtain
	$2^{jn/2}\bigl(\sum_{k\in\Lambda}|c_{j,k}|^pw(Q_{j,k})\bigr)^{1/p}\le K_L[w]_{A_p}^{1/p}\|r^{(j)}_\Lambda\|_{L^p_w}$.
	Since $[w]_{A_p}^{-1/p}\ge[w]_{A_p}^{-\theta(p)}$, this is the lower
	estimate in \eqref{eq:single-scale-lemma} with $c$ equal to the reciprocal
	of the quantity obtained from the right-hand side of \eqref{eq:sse-KL} by
	replacing $\mathrm{L}_\sigma(\tilde{\psi})$ with
	$2^{\sigma}K_\Phi|\tilde{\psi}|_{\mathcal{D}_\sigma}$. Neither $C$
	nor $c$ depends on $j$, $\Lambda$, or $\{c_{j,k}\}$, which completes the
	proof.
\end{proof}

\begin{remark}[Refined powers and three consequences of the proof]
	\label{rem:single-scale-two-sided}
	(a) The proof establishes \eqref{eq:single-scale-lemma} in the form
	\begin{align}\label{eq:sse-refined}
		&K_L^{-1}[w]_{A_p}^{-1/p}\,2^{jn/2}
		\left(\sum_{k\in\Lambda}|c_{j,k}|^p\,w(Q_{j,k})\right)^{1/p}\nonumber\\
		&\le\big\|r^{(j)}_\Lambda\big\|_{L^p_w}\le K_U[w]_{A_p}^{1/p}\,2^{jn/2}
		\left(\sum_{k\in\Lambda}|c_{j,k}|^p\,w(Q_{j,k})\right)^{1/p},
	\end{align}
	with $K_U$ and $K_L$ bounded in \eqref{eq:sse-KU} and \eqref{eq:sse-KL}.
	Since $1/p\le\theta(p)$ with equality exactly at $p=2$,
	\eqref{eq:single-scale-lemma} is a consequence of \eqref{eq:sse-refined}.
	The arguments below use \eqref{eq:sse-refined}.
	(b) Step 2 of the proof uses the generator only through its decay
	envelope $(\beta_m)$. Consequently, for every
	$h\in\mathcal{S}(\mathbb{R}^n)$, with $h_{j,k}(x):=2^{jn/2}h(2^jx-k)$,
	\begin{equation}\label{eq:synthesis-any-generator}
		\Bigl\|\sum_{k\in\Lambda}c_{j,k}\,h_{j,k}\Bigr\|_{L^p_w}
		\le K_U(h)\,[w]_{A_p}^{1/p}\,2^{jn/2}
		\left(\sum_{k\in\Lambda}|c_{j,k}|^p\,w(Q_{j,k})\right)^{1/p},
	\end{equation}
	where $K_U(h)$ is defined by \eqref{eq:sse-upper-refined} with $h$ in
	place of $\psi$ and is bounded by \eqref{eq:sse-KU} with $|h|_{M^\tau}$ in
	place of $|\psi|_{M^\tau}$.
	(c) Let $h$ be a measurable function with
	$|h(x)|\le\mathrm{L}\,(1+|x|)^{-\sigma}$ for some $\sigma>2n$. Inequality
	\eqref{eq:sse-single-k} with $k=0$ and $h$ in place of $\tilde{\psi}$,
	combined with $(2i+1)^n\le(2i+1)^{n+n/p}$ and
	$G_{0,i}\le\|g\|_{L^p_w}$, gives
	\[
	\int_{\mathbb{R}^n}|g|\,|h|\le K_L(h)\,[w]_{A_p}^{1/p}\,w(Q_{0,0})^{-1/p}\,\|g\|_{L^p_w},
	\qquad g\in L^p_w,
	\]
	with $K_L(h)$ defined as in \eqref{eq:sse-analysis-zero} from the
	envelope of $h$, and the argument that led to \eqref{eq:sse-KL} gives
	$K_L(h)\le\mathrm{L}\bigl(1+2\gamma'\,3^{\gamma'-1}\zeta(1+\sigma-n-n/p)\bigr)$
	with $\gamma'=n+n/p$. Applied to Schwartz functions, this shows that every
	$g\in L^p_w$ defines a tempered distribution. It also shows that the
	coefficients $\langle g,\tilde{\psi}_{j,k}\rangle$ are given by absolutely
	convergent integrals.
	(d) Let $s>0$, $0<q<\infty$, $\ell\in\mathbb{Z}$, and let
	$r=\sum_{k\in\Lambda}c_k\,\psi_{\ell,k}$ be a finite single-scale sum. Then
	\begin{align}\label{eq:besov-single-scale}
		&C_M^{-1}\,2^{\ell s}\,\|r\|_{L^p_w}
		\le\|r\|_{\dot{B}^s_{p,q}(w)}
		\le C'_M\,K_\eta\,[w]_{A_p}^{1/p}\,2^{\ell(s+n/2)}
		\left(\sum_{k\in\Lambda}|c_k|^p\,w(Q_{\ell,k})\right)^{1/p},\\
		&K_\eta:=\max_{|i|\le M}K_U(\varphi_i*\psi).\nonumber
	\end{align}
	Indeed, $\hat{r}$ is supported in $\{2^\ell a\le|\xi|\le 2^\ell b\}$, so
	the finite-band property of \eqref{eq:finite-band} gives
	$r=\sum_{|i|\le M}\Delta_{\ell+i}r$ and $\Delta_{j'}r=0$ for
	$|j'-\ell|>M$. The left inequality follows from the triangle inequality
	and $2^{j's}\|\Delta_{j'}r\|_{L^p_w}\le\|r\|_{\dot{B}^s_{p,q}(w)}$. For
	the right inequality, the change of variables
	$u=2^\ell y-k$ gives the dilation identity
	$\varphi_{\ell+i}*\psi_{\ell,k}=(\varphi_i*\psi)_{\ell,k}$. Hence
	$\Delta_{\ell+i}r$ is a single-scale sum with the Schwartz generator
	$\varphi_i*\psi$, and \eqref{eq:synthesis-any-generator} applies. Finally,
	$\|r\|_{\dot{B}^s_{p,q}(w)}\le 2^{\ell s}C'_M\max_{|i|\le M}\|\Delta_{\ell+i}r\|_{L^p_w}$.
\end{remark}

The next lemma supplies, with explicit dependence on the weight, the bound
of the wavelet coefficients by the Besov norm that the Jackson-type
inequality requires.

\begin{lemma}[Coefficient bound with explicit weight dependence]
	\label{lem:coefficient-bound}
	Let Assumption~\ref{ass:wavelet-system} hold with
	$\tilde{\psi}^{(\epsilon)}\in\mathcal{D}_\sigma$ for some $\sigma>2n$, and
	let $s>0$, $1<p<\infty$, $0<q<\infty$, and $w\in A_p$. For
	$f\in L^p_w\cap\dot{B}^s_{p,q}(w)$ and $j\in\mathbb{Z}$ put
	\begin{equation}\label{eq:Sj-def}
		S_j(f):=2^{jn/2}\left(\sum_{\epsilon=1}^{L}\sum_{k\in\mathbb{Z}^n}
		|\langle f,\tilde{\psi}^{(\epsilon)}_{j,k}\rangle|^p\,w(Q_{j,k})\right)^{1/p}.
	\end{equation}
	Then, with $K_L:=\max_{\epsilon}K_L(\tilde{\psi}^{(\epsilon)})$ as in
	\eqref{eq:sse-analysis-zero},
	\begin{equation}\label{eq:coefficient-bound}
		S_j(f)\le L^{1/p}\,C_M\,K_L\,[w]_{A_p}^{1/p}\,2^{-js}\,\|f\|_{\dot{B}^s_{p,q}(w)},
		\qquad j\in\mathbb{Z}.
	\end{equation}
\end{lemma}

\begin{proof}[Proof of Lemma~\ref{lem:coefficient-bound}.]
	We fix $\epsilon$, $j$ and $k$, and we write $\tilde{\psi}$ for
	$\tilde{\psi}^{(\epsilon)}$. The Fourier transform of
	$\tilde{\psi}_{j,k}$ is supported in $\{2^ja\le|\xi|\le 2^jb\}$, so the
	finite-band property of \eqref{eq:finite-band} gives
	$\tilde{\psi}_{j,k}=\sum_{|i|\le M}\varphi_{j+i}*\tilde{\psi}_{j,k}$. The
	function $f$ belongs to $L^p_w$, and by
	Remark~\ref{rem:single-scale-two-sided}(c) its pairings with rapidly
	decreasing functions are absolutely convergent. Since $\varphi$ is
	real-valued and even (Definition~\ref{def:littlewood-paley}), Fubini's
	theorem gives
	$\langle f,\varphi_{j'}*\tilde{\psi}_{j,k}\rangle=\langle\varphi_{j'}*f,\tilde{\psi}_{j,k}\rangle$.
	Its application is justified because $|\varphi_{j'}|*|\tilde{\psi}_{j,k}|$
	is rapidly decreasing. Since $\varphi_{j'}*f=\Delta_{j'}f$,
	\begin{equation}\label{eq:coef-identity}
		\langle f,\tilde{\psi}_{j,k}\rangle
		=\sum_{|i|\le M}\langle\Delta_{j+i}f,\tilde{\psi}_{j,k}\rangle .
	\end{equation}
	Each $\Delta_{j+i}f$ belongs to $L^p_w$ because
	$2^{(j+i)s}\|\Delta_{j+i}f\|_{L^p_w}\le\|f\|_{\dot{B}^s_{p,q}(w)}$
	(Definition~\ref{def:weighted-besov}). We take the weighted $\ell^p$ norm
	in $k$ in \eqref{eq:coef-identity}, apply Minkowski's inequality in $i$,
	and apply \eqref{eq:sse-analysis-general} to each
	$g=\Delta_{j+i}f$. This gives
	\begin{align*}
		2^{jn/2}\left(\sum_{k}|\langle f,\tilde{\psi}_{j,k}\rangle|^p\,w(Q_{j,k})\right)^{1/p}
		&\le K_L[w]_{A_p}^{1/p}\sum_{|i|\le M}\|\Delta_{j+i}f\|_{L^p_w}\\
		&\le K_L[w]_{A_p}^{1/p}\,\|f\|_{\dot{B}^s_{p,q}(w)}\sum_{|i|\le M}2^{-(j+i)s}.
	\end{align*}
	The last sum equals $C_M2^{-js}$. Summing the $p$-th powers over the $L$
	generator labels contributes the factor $L^{1/p}$, and
	\eqref{eq:coefficient-bound} follows.
\end{proof}

Inequality \eqref{eq:coefficient-bound} is a scale-by-scale bound of the
kind provided qualitatively by the boundedness of the analysis operator in
\cite[Theorem~3.5]{Bownik2005}. Its constant is explicit,
carries the power $[w]_{A_p}^{1/p}$, and does not depend on $q$, so the inequality holds with the same constant for every $0<q<\infty$.

We now introduce the class on which the Jackson-type inequality holds.

\begin{definition}[Localized class]\label{def:localized-class}
	Let $Q_0:=[0,1)^n$ and
	$\mathcal{D}(Q_0):=\{(j,k)\in\mathbb{Z}\times\mathbb{Z}^n : j\ge 0,\ Q_{j,k}\subset Q_0\}$.
	For each $j\ge 0$ exactly $2^{jn}$ indices $k$ satisfy
	$(j,k)\in\mathcal{D}(Q_0)$, namely those with $0\le k_i\le 2^j-1$ for
	$1\le i\le n$. Let $1<p<\infty$ and $w\in A_p$. For $f\in L^p_w$ and an integer $J\ge 0$ put
	\begin{equation}\label{eq:tree-partial-sum}
		T_Jf:=\sum_{\epsilon=1}^{L}\sum_{j=0}^{J}\sum_{k:\,(j,k)\in\mathcal{D}(Q_0)}
		\langle f,\tilde{\psi}^{(\epsilon)}_{j,k}\rangle\,\psi^{(\epsilon)}_{j,k},
		\qquad T_{-1}f:=0,
	\end{equation}
	a finite sum whose coefficients are well defined by
	Remark~\ref{rem:single-scale-two-sided}(c). For $s>0$ and $0<q<\infty$, the localized class
	$\dot{B}^s_{p,q}(w;Q_0)$ consists of all $f\in L^p_w\cap\dot{B}^s_{p,q}(w)$
	such that $T_Jf\to f$ in $L^p_w$ as $J\to\infty$.
\end{definition}

\begin{remark}[On the localized class]\label{rem:localized-class}
	Every finite sum $f=\sum_{\epsilon}\sum_{(j,k)\in\Lambda}c^{(\epsilon)}_{j,k}\psi^{(\epsilon)}_{j,k}$
	with a finite set $\Lambda\subset\mathcal{D}(Q_0)$ belongs to
	$\dot{B}^s_{p,q}(w;Q_0)$. Such an $f$ belongs to $L^p_w$ by
	\eqref{eq:synthesis-any-generator}, and to $\dot{B}^s_{p,q}(w)$ by
	\eqref{eq:besov-single-scale}, applied to each of its finitely many single-scale parts with a fixed label and combined by the triangle inequality in $L^p_w$ and the quasi-triangle inequality in $\dot{B}^s_{p,q}(w)$. By \eqref{eq:biorthogonality},
	$T_Jf=f$ for all $J$ exceeding the largest generation in $\Lambda$. In
	particular the extremal family of Corollary~\ref{lem:lower-bound-constant}
	lies in the class.
\end{remark}

\begin{theorem}[Jackson-type inequality for $N$-term wavelet approximation]
	\label{thm:main1}
	Let Assumption~\ref{ass:wavelet-system} hold, with
	$\psi^{(\epsilon)}\in M^\tau$ and
	$\tilde{\psi}^{(\epsilon)}\in\mathcal{D}_\sigma$ for some $\tau>0$ and
	$\sigma>2n$. Let $s>0$, $1<p<\infty$, $0<q<\infty$, and $w\in A_p$. Then
	for every $f\in\dot{B}^s_{p,q}(w;Q_0)$ and every integer $N\ge 1$, the
	$N$-term approximation error \eqref{eq:nterm-error} relative to the
	system $\mathcal{F}=\{\psi_{j,k}\}$ satisfies
	\begin{equation}\label{eq:jackson-main}
		E_N^{\mathcal{F}}(f)_{L^p_w}
		\le C_0\,[w]_{A_p}^{2/p}\,N^{-s/n}\,\|f\|_{\dot{B}^s_{p,q}(w)}
		\le C_0\,[w]_{A_p}^{2\theta(p)}\,N^{-s/n}\,\|f\|_{\dot{B}^s_{p,q}(w)},
	\end{equation}
	where $\theta(p)$ is given by \eqref{eq:single-scale-theta} and
	\begin{equation}\label{eq:constant-form}
		C_0:=\frac{L\,(2L)^{s/n}}{1-2^{-s}}\,C_M\,K_U\,K_L,
		\qquad
		K_U:=\max_{\epsilon}K_U(\psi^{(\epsilon)}),
		\quad
		K_L:=\max_{\epsilon}K_L(\tilde{\psi}^{(\epsilon)}).
	\end{equation}
	Here $K_U$ and $K_L$ are the decay quantities of
	\eqref{eq:sse-upper-refined} and \eqref{eq:sse-analysis-zero}, bounded
	in closed form by \eqref{eq:sse-KU} and \eqref{eq:sse-KL}, and $C_M$ is
	given by \eqref{eq:finite-band}. The factor $C_0$ is determined by $n$, $L$, $p$, $s$, $M$, and the decay
	envelopes of the generators $\psi^{(\epsilon)}$ and $\tilde{\psi}^{(\epsilon)}$
	through the series $K_U$ and $K_L$. Replacing $K_U$ and $K_L$ by the
	maxima over $\epsilon$ of the right-hand sides of \eqref{eq:sse-KU} and
	\eqref{eq:sse-KL}, and $\mathrm{L}_\sigma(\tilde{\psi}^{(\epsilon)})$ by its
	bound in \eqref{eq:decay-pointwise}, yields a closed-form majorant of $C_0$
	that depends only on $n$, $L$, $p$, $s$, $M$, and the decay parameters
	$\tau,\sigma$ in the sense of Definition~\ref{def:decay-classes}. Both
	$C_0$ and this majorant are independent of $N$, $f$, $q$, $w$, and the Riesz
	bounds $A,B$.
\end{theorem}

\begin{proof}[Proof of Theorem~\ref{thm:main1}.]
	We construct an admissible approximant by truncation in scale along the
	dyadic tree of $Q_0$. The residual is then estimated generation by
	generation through \eqref{eq:sse-refined} and
	Lemma~\ref{lem:coefficient-bound}. No greedy selection is used.
	
	\emph{Step 1: the approximant.} We put
	\begin{equation}\label{eq:scale-choice}
		j_0:=\max\Bigl\{j\ge -1 : L\,\frac{2^{(j+1)n}-1}{2^n-1}\le N\Bigr\}.
	\end{equation}
	The set is nonempty because $j=-1$ gives $0\le N$, and it is bounded above
	because the left-hand side tends to infinity with $j$. We define
	$g_N:=T_{j_0}f$ by \eqref{eq:tree-partial-sum}. It is a combination of
	$L\sum_{j=0}^{j_0}2^{jn}=L(2^{(j_0+1)n}-1)/(2^n-1)\le N$ elements of
	$\mathcal{F}$, so it is admissible in \eqref{eq:nterm-error} and
	$E_N^{\mathcal{F}}(f)_{L^p_w}\le\|f-g_N\|_{L^p_w}$. By maximality of
	$j_0$ we have $L(2^{(j_0+2)n}-1)/(2^n-1)>N$. Since $2^n-1\ge 2^{n-1}$, the
	left-hand side is smaller than $2L\cdot 2^{(j_0+1)n}$. Hence
	$2^{(j_0+1)n}>N/(2L)$, and therefore
	\begin{equation}\label{eq:rate}
		2^{-(j_0+1)s}=\bigl(2^{(j_0+1)n}\bigr)^{-s/n}<(2L)^{s/n}\,N^{-s/n}.
	\end{equation}
	
	\emph{Step 2: the residual.} For $j\ge 0$ put
	$c^{(\epsilon)}_{j,k}:=\langle f,\tilde{\psi}^{(\epsilon)}_{j,k}\rangle$
	and
	$r^{(j)}:=\sum_{\epsilon}\sum_{k:\,(j,k)\in\mathcal{D}(Q_0)}c^{(\epsilon)}_{j,k}\psi^{(\epsilon)}_{j,k}$.
	Since $f\in\dot{B}^s_{p,q}(w;Q_0)$, we have
	$f-g_N=\lim_{J\to\infty}(T_Jf-T_{j_0}f)=\lim_{J\to\infty}\sum_{j=j_0+1}^{J}r^{(j)}$
	in $L^p_w$. The triangle inequality and the continuity of the norm give
	\begin{equation}\label{eq:triangle}
		\|f-g_N\|_{L^p_w}\le\sum_{j>j_0}\big\|r^{(j)}\big\|_{L^p_w}.
	\end{equation}
	
	\emph{Step 3: one generation.} Fix $j>j_0$, so $j\ge 0$. For each
	$\epsilon$, the part of $r^{(j)}$ with label $\epsilon$ is a single-scale
	sum over the finite set $\{k:(j,k)\in\mathcal{D}(Q_0)\}$. The upper
	estimate in \eqref{eq:sse-refined} applies to it. Summing over $\epsilon$
	and applying H\"older's inequality in $\epsilon$ gives
	\begin{equation}\label{eq:single-scale-applied}
		\big\|r^{(j)}\big\|_{L^p_w}
		\le K_U[w]_{A_p}^{1/p}\,L^{1/p'}\,S_j(f)
		\le L\,C_M\,K_U\,K_L\,[w]_{A_p}^{2/p}\,2^{-js}\,\|f\|_{\dot{B}^s_{p,q}(w)}.
	\end{equation}
	The first inequality uses that the sum over the tree indices is a subsum
	of the sum defining $S_j(f)$ in \eqref{eq:Sj-def}. The second inequality
	is \eqref{eq:coefficient-bound}, and $L^{1/p'}L^{1/p}=L$.
	
	\emph{Step 4: summation.} Inserting \eqref{eq:single-scale-applied} into
	\eqref{eq:triangle} and summing the geometric series, which converges
	because $s>0$, we obtain
	\[
	\|f-g_N\|_{L^p_w}
	\le L\,C_MK_UK_L\,[w]_{A_p}^{2/p}\,\|f\|_{\dot{B}^s_{p,q}(w)}\,
	\frac{2^{-(j_0+1)s}}{1-2^{-s}}.
	\]
	By \eqref{eq:rate} the right-hand side is at most
	$C_0[w]_{A_p}^{2/p}N^{-s/n}\|f\|_{\dot{B}^s_{p,q}(w)}$ with $C_0$ as in
	\eqref{eq:constant-form}. This proves the first inequality in
	\eqref{eq:jackson-main}. The second follows from $[w]_{A_p}\ge 1$ and
	$2/p\le 2\theta(p)$. The power $[w]_{A_p}^{2/p}$ is the product of the
	synthesis contribution in \eqref{eq:sse-refined} and the coefficient
	contribution in \eqref{eq:coefficient-bound}, each of which carries the
	power $[w]_{A_p}^{1/p}$.
\end{proof}

\begin{corollary}[Lower estimate for the approximation constant via an extremal function family]
	\label{lem:lower-bound-constant}
	Under the hypotheses of Theorem~\ref{thm:main1}, let $\psi:=\psi^{(1)}$,
	$\tilde{\psi}:=\tilde{\psi}^{(1)}$, and let $K_U$, $K_L$, $K_\eta$ be the
	constants of \eqref{eq:sse-upper-refined}, \eqref{eq:sse-analysis-zero}
	and \eqref{eq:besov-single-scale} for this pair. For every integer
	$m\ge 0$ and $N:=2^{mn}$ there exists a finite wavelet sum $f^{*}_N$ of
	generation $m+1$, belonging to $\dot{B}^s_{p,q}(w;Q_0)$, such that
	\begin{equation}\label{eq:lbc-twosided}
		c_\flat\,[w]_{A_p}^{-2/p}
		\le
		\frac{E_N^{\mathcal{F}}(f^{*}_N)_{L^p_w}}
		{N^{-s/n}\,\|f^{*}_N\|_{\dot{B}^s_{p,q}(w)}}
		\le
		c_\sharp\,[w]_{A_p}^{2/p},
	\end{equation}
	where
	\begin{equation}\label{eq:lbc-constants}
		c_\flat:=\frac{(1-2^{-n})^{1/p}\,2^{-s}}{K_L\,K_\eta\,C'_M},
		\qquad
		c_\sharp:=(1-2^{-n})^{1/p}\,2^{-s}\,K_U\,K_L\,C_M
	\end{equation}
	are independent of $N$ and $w$. Since $[w]_{A_p}\ge 1$ and
	$2/p\le 2\theta(p)$, the ratio in \eqref{eq:lbc-twosided} also lies
	between $c_\flat[w]_{A_p}^{-2\theta(p)}$ and
	$c_\sharp[w]_{A_p}^{2\theta(p)}$. In particular, for fixed $w$ the
	exponent $s/n$ in \eqref{eq:jackson-main} cannot be replaced by a larger
	exponent on $\dot{B}^s_{p,q}(w;Q_0)$.
\end{corollary}

\begin{proof}[Proof of Corollary~\ref{lem:lower-bound-constant}.]
	We write $N=2^{mn}$ and build the extremal function at generation
	$\ell:=m+1$. Let $K:=\{k\in\mathbb{Z}^n : 0\le k_i<2^\ell\}$, so that
	$\#K=2^{\ell n}=2^nN$ and $(\ell,k)\in\mathcal{D}(Q_0)$ for every
	$k\in K$. We set
	\begin{equation}\label{eq:lbc-extremal}
		f^{*}_N:=\sum_{k\in K}a_k\,\psi_{\ell,k},
		\qquad
		a_k:=2^{-\ell(s+n/2)}\,w(Q_{\ell,k})^{-1/p}.
	\end{equation}
	Each $w(Q_{\ell,k})$ is finite and positive, so $a_k$ is well defined, and
	$f^{*}_N\in\dot{B}^s_{p,q}(w;Q_0)$ by
	Remark~\ref{rem:localized-class}. The normalization is matched to the
	exponent $p$, so that
	\begin{equation}\label{eq:lbc-summand}
		|a_k|^p\,w(Q_{\ell,k})=2^{-\ell p(s+n/2)},
		\qquad
		2^{\ell n/2}\left(\sum_{k\in K'}|a_k|^p\,w(Q_{\ell,k})\right)^{1/p}
		=(\#K')^{1/p}\,2^{-\ell s}
	\end{equation}
	for every $K'\subset K$. The summand is independent of $k$ and of $w$.
	
	The Besov norm of $f^{*}_N$ is controlled by \eqref{eq:besov-single-scale}
	and by the lower estimate in \eqref{eq:sse-refined}. Together with
	\eqref{eq:lbc-summand} for $K'=K$ they give
	\begin{equation}\label{eq:lbc-besov}
		C_M^{-1}K_L^{-1}[w]_{A_p}^{-1/p}\,(2^nN)^{1/p}
		\le\|f^{*}_N\|_{\dot{B}^s_{p,q}(w)}
		\le C'_M\,K_\eta\,[w]_{A_p}^{1/p}\,(2^nN)^{1/p}.
	\end{equation}
	
	For the upper bound on the error, fix $K_0\subset K$ with $\#K_0=N$ and
	put $g_0:=\sum_{k\in K_0}a_k\psi_{\ell,k}$. The residual
	$f^{*}_N-g_0$ is a single-scale sum over $(2^n-1)N$ indices, so the upper
	estimate in \eqref{eq:sse-refined} and \eqref{eq:lbc-summand} give
	\begin{equation}\label{eq:lbc-upper}
		E_N^{\mathcal{F}}(f^{*}_N)_{L^p_w}
		\le\|f^{*}_N-g_0\|_{L^p_w}
		\le K_U[w]_{A_p}^{1/p}\bigl((2^n-1)N\bigr)^{1/p}\,2^{-\ell s}.
	\end{equation}
	
	For the lower bound, let
	$g=\sum_{(\epsilon,j,k)\in\Lambda}d^{(\epsilon)}_{j,k}\psi^{(\epsilon)}_{j,k}$
	with $\#\Lambda\le N$, where $\epsilon$ ranges over all labels of
	Assumption~\ref{ass:wavelet-system}. By
	\eqref{eq:biorthogonality}, the coefficient
	$b_k:=\langle f^{*}_N-g,\tilde{\psi}_{\ell,k}\rangle$ equals $a_k$ for
	every $k\in K$ with $(1,\ell,k)\notin\Lambda$. There are at least $\#K-\#\Lambda\ge(2^n-1)N$ such
	indices. We apply \eqref{eq:sse-analysis-general} to $f^{*}_N-g$ and
	discard the remaining indices. Together with \eqref{eq:lbc-summand} this
	gives
	\begin{equation}\label{eq:lbc-lower-step}
		K_L[w]_{A_p}^{1/p}\,\|f^{*}_N-g\|_{L^p_w}
		\ge 2^{\ell n/2}\left(\sum_{k\in K}|b_k|^p\,w(Q_{\ell,k})\right)^{1/p}
		\ge\left[(2^n-1)N\right]^{1/p}\,2^{-\ell s}.
	\end{equation}
	Taking the infimum over $g$ yields
	$E_N^{\mathcal{F}}(f^{*}_N)_{L^p_w}\ge K_L^{-1}[w]_{A_p}^{-1/p}((2^n-1)N)^{1/p}2^{-\ell s}$.
	
	Since $2^{\ell n}=2^nN$, we have $2^{-\ell s}=2^{-s}N^{-s/n}$. Dividing
	\eqref{eq:lbc-upper} and \eqref{eq:lbc-lower-step} by
	$N^{-s/n}\|f^{*}_N\|_{\dot{B}^s_{p,q}(w)}$ and inserting
	\eqref{eq:lbc-besov}, the powers of $N$ cancel through
	\begin{equation}\label{eq:lbc-ratio}
		\frac{\bigl((2^n-1)N\bigr)^{1/p}\,2^{-s}N^{-s/n}}{N^{-s/n}\,(2^nN)^{1/p}}
		=\bigl(1-2^{-n}\bigr)^{1/p}\,2^{-s},
	\end{equation}
	and \eqref{eq:lbc-twosided} follows with the constants
	\eqref{eq:lbc-constants}. Finally, suppose that
	$E_N^{\mathcal{F}}(f)_{L^p_w}\le C'N^{-s'/n}\|f\|_{\dot{B}^s_{p,q}(w)}$
	held on $\dot{B}^s_{p,q}(w;Q_0)$ for some $s'>s$ and some $C'$
	independent of $N$ and $f$. Then
	\eqref{eq:lbc-twosided} would give
	$c_\flat[w]_{A_p}^{-2/p}\le C'N^{-(s'-s)/n}$ for every $N=2^{mn}$, which
	fails as $m\to\infty$.
\end{proof}

\begin{remark}[The localization cannot be removed]
	\label{rem:localization-necessary}
	Let $w\equiv 1$ and let $\psi:=\psi^{(1)}$. For $N\ge 1$ put
	$f_N:=\sum_{k\in K_N}\psi_{0,k}$ with
	$K_N:=\{1,\dots,2N\}\times\{0\}^{n-1}$. This is a finite single-scale sum
	of generation $0$ in $L^p\cap\dot{B}^s_{p,q}$, and it is not carried by the
	dyadic tree of $Q_0$. For every $g$ with at most $N$ terms, at least $N$
	of the $2N$ coefficients $\langle f_N-g,\tilde{\psi}_{0,k}\rangle$,
	$k\in K_N$, are equal to one. Hence \eqref{eq:sse-analysis-general} gives
	$E_N^{\mathcal{F}}(f_N)_{L^p}\ge K_L^{-1}N^{1/p}$, while
	\eqref{eq:besov-single-scale} gives
	$\|f_N\|_{\dot{B}^s_{p,q}}\le C'_MK_\eta(2N)^{1/p}$. Therefore
	$E_N^{\mathcal{F}}(f_N)_{L^p}\ge(2^{1/p}K_LK_\eta C'_M)^{-1}\|f_N\|_{\dot{B}^s_{p,q}}$
	for every $N$. Consequently, for $w\equiv 1$, no inequality of the form
	$E_N^{\mathcal{F}}(f)_{L^p}\le CN^{-\kappa}\|f\|_{\dot{B}^s_{p,q}}$
	with $\kappa>0$ and $C$ independent of $N$ and $f$ can hold on the whole
	of $L^p\cap\dot{B}^s_{p,q}$. For $n=1$, $p=q=2$ and the orthonormal
	Meyer system of Example~\ref{ex:meyer-systems}, the same family gives
	$E_N^{\mathcal{F}}(f_N)_{L^2}=2^{-1/2}\|f_N\|_{L^2}$ exactly.
\end{remark}

\section{The Inverse Machinery: Bernstein Inequalities and the Weighted Approximation Scale}\label{sec4}

Throughout this section Assumption~\ref{ass:mra} is in force in addition to
Assumption~\ref{ass:wavelet-system}, and $\chi$, $\tilde{\chi}$ and $R$ are as
there. The constants $K_U$ and $K_L$ are those of \eqref{eq:constant-form},
namely the maxima over $\epsilon$ of the series $K_U(\psi^{(\epsilon)})$ of
\eqref{eq:sse-upper-refined} and $K_L(\tilde{\psi}^{(\epsilon)})$ of
\eqref{eq:sse-analysis-zero}. For $1 < p < \infty$ and $w \in A_p$, every $g \in L^p_w$ defines a tempered
distribution, and its pairings
$\langle g, h \rangle := \int_{\mathbb{R}^n} g\, \overline{h}$ with
$h \in \mathcal{S}(\mathbb{R}^n)$ are absolutely convergent integrals
(Remark~\ref{rem:single-scale-two-sided}(c)). Accordingly, and as in Section~\ref{sec3},
$L^p_w \cap \dot{B}^s_{p,q}(w)$ denotes the set of $g \in L^p_w$ whose
class modulo $\mathcal{P}$ belongs to $\dot{B}^s_{p,q}(w)$.

\begin{definition}[Multiresolution subspaces and the linear approximation
	error]
	\label{def:b-multiresolution}
	Let $1 < p < \infty$ and $w \in A_p$. For $m \in \mathbb{Z}$ let
	\begin{equation}\label{eq:b-mra-space}
		V_m := \overline{\operatorname{span}}
		\bigl\{ \chi_{m,k} :\, k \in \mathbb{Z}^n \bigr\},
	\end{equation}
	the closure being taken in $L^p_w$, and for $f \in L^p_w$ let
	\begin{equation}\label{eq:b-linear-error}
		E(f, V_m)_{L^p_w} := \inf_{g \in V_m} \|f - g\|_{L^p_w}.
	\end{equation}
\end{definition}

The space $V_m$ depends on $p$ and $w$, and \eqref{eq:b-mra-space} is
meaningful because $\mathcal{S}(\mathbb{R}^n) \subset L^p_w$. Indeed, let
$h \in \mathcal{S}(\mathbb{R}^n)$, fix $N > n + 1/p$, put
$C_N := \sup_{x} (1 + |x|)^N |h(x)|$, and let $Q^{(i)} := [-i, i+1)^n$ for
$i \ge 0$, so that $Q^{(0)} = Q_{0,0}$. Inequality
\eqref{eq:ap-doubling}, applied to the cube $Q^{(i)}$ and the set
$Q_{0,0} \subset Q^{(i)}$, gives
$w(Q^{(i)}) \le [w]_{A_p}\, (2i+1)^{np}\, w(Q_{0,0})$. Every
$x \in Q^{(i)} \setminus Q^{(i-1)}$, $i \ge 1$, has a coordinate of modulus
at least $i - 1$, so that $|h(x)| \le C_N (1 + |x|)^{-N} \le C_N i^{-N}$
there. Hence
\[
\|h\|_{L^p_w}^p
\le C_N^p\, w(Q_{0,0})
\Bigl( 1 + [w]_{A_p} \sum_{i \ge 1} i^{-Np}\, (2i+1)^{np} \Bigr),
\]
and the series converges because $i^{-Np} (2i+1)^{np} \le 3^{np}\,
i^{-(N-n)p}$ and $(N - n)p > 1$.

The estimates of this section rest on a weighted bound for integral
operators whose kernels are concentrated at a single dyadic scale.
Lemma~\ref{lem:Bern:projection-bounded} records this bound together with
its consequences for the Littlewood--Paley operators and for the
multiresolution projections. Theorem~\ref{thm:B1} is the single-scale
Bernstein estimate, Theorem~\ref{thm:B2} the Bernstein inequality on
$V_m$, Lemma~\ref{lem:Bern:linear-jackson} the linear Jackson estimate,
and Theorem~\ref{thm:B3} the identification of the linear approximation
scale.

	\begin{lemma}[Single-scale kernel operators and multiresolution projections]
		\label{lem:Bern:projection-bounded}
		Let $1 < p < \infty$ and $w \in A_p$.
		\begin{enumerate}
			\item[(i)] Let $j \in \mathbb{Z}$, let
			$\kappa : [0,\infty) \to [0,\infty)$ be nonincreasing with
			$\kappa(t) \le C_\kappa (1 + t)^{-N}$ for some $C_\kappa > 0$ and
			$N > n + n/p$, and let $\mathcal{K}$ be a measurable function on
			$\mathbb{R}^n \times \mathbb{R}^n$ with
			$|\mathcal{K}(x,y)| \le 2^{jn} \kappa(2^j |x - y|_\infty)$. For
			every $g \in L^p_w$ the integral
			$Tg(x) := \int_{\mathbb{R}^n} \mathcal{K}(x,y)\, g(y)\, dy$
			converges absolutely for almost every $x$, and
			\begin{equation}\label{eq:b-kernel-bound}
				\|Tg\|_{L^p_w} \le K(\kappa)\, [w]_{A_p}^{1/p}\, \|g\|_{L^p_w},
				\qquad
				K(\kappa) := \sum_{i \ge 0} (\kappa_i - \kappa_{i+1})\,
				(2i+1)^{n + n/p},
			\end{equation}
			where $\kappa_0 := \kappa(0)$ and $\kappa_i := \kappa(i-1)$ for
			$i \ge 1$; the series $K(\kappa)$ is finite.
			\item[(ii)] Let
			$\kappa_\Delta(t) := \sup\{ |\varphi(z)| : |z|_\infty \ge t \}$,
			$\Pi(u,v) := \sum_{k \in \mathbb{Z}^n} \chi(u - k)\,
			\overline{\tilde{\chi}(v - k)}$ and
			$\kappa_P(t) := \sup\{ |\Pi(u,v)| : |u - v|_\infty \ge t \}$, and
			put $K_\Delta := K(\kappa_\Delta)$ and $K_P := K(\kappa_P)$, both
			finite. For $j, m \in \mathbb{Z}$ and $g \in L^p_w$ let
			$\Delta_j g := \varphi_j * g$ and
			\begin{equation}\label{eq:b-partial-reconstruction}
				P_m g := \sum_{k \in \mathbb{Z}^n}
				\langle g, \tilde{\chi}_{m,k} \rangle\, \chi_{m,k},
			\end{equation}
			the series converging unconditionally in $L^p_w$. Then
			\begin{equation}\label{eq:b-delta-bound}
				\|\Delta_j g\|_{L^p_w} \le K_\Delta\, [w]_{A_p}^{1/p}\,
				\|g\|_{L^p_w}
			\end{equation}
			and
			\begin{equation}\label{eq:b-projection-bounded}
				\|P_m g\|_{L^p_w} \le K_P\, [w]_{A_p}^{1/p}\, \|g\|_{L^p_w}.
			\end{equation}
			\item[(iii)] For every $m \in \mathbb{Z}$, $P_m$ maps $L^p_w$ onto
			$V_m$, and $P_m g = g$ for every $g \in V_m$.
			\item[(iv)] For every $g \in L^p_w$,
			\begin{equation}\label{eq:b-projection-limit}
				\lim_{J \to \infty} \|P_J g - g\|_{L^p_w} = 0 .
			\end{equation}
			\item[(v)] For $g \in L^p_w$ and $j \in \mathbb{Z}$ the series
			$r^{(j)}(g) := \sum_{\epsilon=1}^{L} \sum_{k \in \mathbb{Z}^n}
			\langle g, \tilde{\psi}^{(\epsilon)}_{j,k} \rangle\,
			\psi^{(\epsilon)}_{j,k}$ converges unconditionally in $L^p_w$, and
			for all integers $m < J$,
			\begin{equation}\label{eq:b-tail-identity}
				P_J g - P_m g = \sum_{j=m}^{J-1} r^{(j)}(g).
			\end{equation}
			\item[(vi)] Let $R$ be as in Assumption~\ref{ass:mra}
			and $M_\chi := \lceil \log_2 (R/c_1) \rceil$. For all
			$m, j' \in \mathbb{Z}$ with $j' \ge m + M_\chi + 1$,
			\begin{equation}\label{eq:b-annihilation}
				\Delta_{j'} g = 0 \qquad \text{for every } g \in V_m .
			\end{equation}
		\end{enumerate}
	\end{lemma}

\begin{proof}[Proof of Lemma~\ref{lem:Bern:projection-bounded}.]
We prove the six assertions in order. For $j\in\mathbb{Z}$, $k\in\mathbb{Z}^n$
and $i\ge 0$ put $Q^{(i)}_{j,k}:=2^{-j}\bigl(k+[-i,i+1)^n\bigr)$ and
$Q^{(-1)}_{j,k}:=\emptyset$. Then $Q^{(0)}_{j,k}=Q_{j,k}$,
$|Q^{(i)}_{j,k}|=2^{-jn}(2i+1)^n$, and the shells
$Q^{(i)}_{j,k}\setminus Q^{(i-1)}_{j,k}$, $i\ge 0$, partition $\mathbb{R}^n$.
Two elementary facts are used repeatedly. First, since $w>0$ almost
everywhere, a sequence converging in $L^p_w$ has a subsequence converging
almost everywhere. Consequently, if a series with terms in
$L^2(\mathbb{R}^n)\cap L^p_w$ converges unconditionally both in
$L^2(\mathbb{R}^n)$ and in $L^p_w$, then its partial sums over an increasing
sequence of finite sets exhausting the index set have a subsequence
converging almost everywhere to each of the two sums, and the two sums
coincide almost everywhere. Second, the set $L^\infty_c$ of bounded
measurable functions vanishing outside a bounded set is contained in
$L^2(\mathbb{R}^n)\cap L^p_w$, because $w$ is locally integrable, and it is
dense in $L^p_w$: for $g\in L^p_w$ the truncations
$g\,\mathbf{1}_{\{|x|\le\lambda,\ |g(x)|\le\lambda\}}$ converge to $g$ in
$L^p_w$ as $\lambda\to\infty$ by dominated convergence.

\smallskip\noindent\emph{(i) The weighted kernel bound.} We first bound
$K(\kappa)$. Put $\gamma:=n+n/p$, $\alpha_i:=(2i+1)^\gamma$ for $i\ge 0$ and
$\alpha_{-1}:=0$. The sequence $(\kappa_i)_{i\ge 0}$ is nonincreasing and
$\kappa_i=\kappa(i-1)\le C_\kappa\, i^{-N}$ for $i\ge 1$, so
$\kappa_i=\sum_{l\ge i}(\kappa_l-\kappa_{l+1})$ with nonnegative terms.
Writing $\alpha_l=\sum_{i=0}^{l}(\alpha_i-\alpha_{i-1})$ and exchanging the
order of summation, which is legitimate because all terms are nonnegative,
we obtain $K(\kappa)=\sum_{i\ge 0}\kappa_i(\alpha_i-\alpha_{i-1})$. For
$i\ge 1$ the mean value theorem and $\gamma>1$ give
$\alpha_i-\alpha_{i-1}\le 2\gamma(2i+1)^{\gamma-1}\le 2\gamma\,3^{\gamma-1}i^{\gamma-1}$.
Since $\kappa_0\le C_\kappa$ and $\alpha_0-\alpha_{-1}=1$,
\[
K(\kappa)\le C_\kappa\Bigl(1+2\gamma\,3^{\gamma-1}\,
\zeta\bigl(1+N-n-n/p\bigr)\Bigr),
\]
where $\zeta$ is the Riemann zeta function; the right-hand side is finite
because $N>n+n/p$.

The decomposition into cube shells below is the one used in Step~3 of the
proof of Lemma~\ref{lem:single-scale-estimate}. By Tonelli's theorem, for
almost every $x$ the function $y\mapsto\mathcal{K}(x,y)$ is measurable and
$x\mapsto\int_{\mathbb{R}^n}|\mathcal{K}(x,y)|\,|g(y)|\,dy$ is measurable.
Fix $k\in\mathbb{Z}^n$ and $x\in Q_{j,k}$. If $i\ge 1$ and
$y\in Q^{(i)}_{j,k}\setminus Q^{(i-1)}_{j,k}$, then some coordinate of
$2^jy-k$ lies in $[-i,-i+1)\cup[i,i+1)$, while the same coordinate of
$2^jx-k$ lies in $[0,1)$; hence $2^j|x-y|_\infty>i-1$ and
$|\mathcal{K}(x,y)|\le 2^{jn}\kappa(i-1)=2^{jn}\kappa_i$. For $i=0$ the bound
$|\mathcal{K}(x,y)|\le 2^{jn}\kappa_0$ holds because $\kappa$ is
nonincreasing. Inserting $\kappa_i=\sum_{l\ge i}(\kappa_l-\kappa_{l+1})$ and
exchanging the order of summation once more, we find
\[
\int_{\mathbb{R}^n}|\mathcal{K}(x,y)|\,|g(y)|\,dy
\le 2^{jn}\sum_{i\ge 0}\kappa_i\int_{Q^{(i)}_{j,k}\setminus Q^{(i-1)}_{j,k}}|g|
=2^{jn}\sum_{l\ge 0}(\kappa_l-\kappa_{l+1})\int_{Q^{(l)}_{j,k}}|g| .
\]
Inequality \eqref{eq:ap-holder} on the cube $Q^{(l)}_{j,k}$,
combined with $w(Q^{(l)}_{j,k})\ge w(Q_{j,k})>0$, gives
$2^{jn}\int_{Q^{(l)}_{j,k}}|g|\le[w]_{A_p}^{1/p}(2l+1)^n\,
w(Q_{j,k})^{-1/p}\,G_{k,l}$, where
$G_{k,l}:=\bigl(\int_{Q^{(l)}_{j,k}}|g|^pw\bigr)^{1/p}$. Consequently, for
almost every $x\in Q_{j,k}$,
\begin{gather*}
\int_{\mathbb{R}^n}|\mathcal{K}(x,y)|\,|g(y)|\,dy
\le[w]_{A_p}^{1/p}\,w(Q_{j,k})^{-1/p}\,H_k,\\
H_k:=\sum_{l\ge 0}(\kappa_l-\kappa_{l+1})(2l+1)^n\,G_{k,l}.
\end{gather*}
Since $G_{k,l}\le\|g\|_{L^p_w}$ and $(2l+1)^n\le(2l+1)^{n+n/p}$, we have
$H_k\le K(\kappa)\|g\|_{L^p_w}<\infty$. Hence the integral defining $Tg(x)$
converges absolutely for almost every $x$, and $Tg$ is measurable, since
Tonelli's theorem applies to the positive and negative parts of the real
and imaginary parts of $\mathcal{K}(x,y)g(y)$. Raising the last display to
the power $p$, integrating against $w$ over $Q_{j,k}$ and summing over $k$,
we obtain $\|Tg\|_{L^p_w}^p\le[w]_{A_p}\sum_{k}H_k^p$. Minkowski's
inequality in $\ell^p(\mathbb{Z}^n)$, applied to the series defining $H_k$,
gives
\[
\Bigl(\sum_{k\in\mathbb{Z}^n}H_k^p\Bigr)^{1/p}
\le\sum_{l\ge 0}(\kappa_l-\kappa_{l+1})(2l+1)^n
\Bigl(\sum_{k\in\mathbb{Z}^n}G_{k,l}^p\Bigr)^{1/p}.
\]
A point $y$ lies in $Q^{(l)}_{j,k}$ exactly when
$2^jy_\iota-l-1<k_\iota\le 2^jy_\iota+l$ for $\iota=1,\dots,n$, and each of
these half-open intervals of length $2l+1$ contains exactly $2l+1$ integers.
Every point of $\mathbb{R}^n$ therefore lies in exactly $(2l+1)^n$ of the
cubes $Q^{(l)}_{j,k}$, $k\in\mathbb{Z}^n$, so that
$\sum_{k}G_{k,l}^p=(2l+1)^n\|g\|_{L^p_w}^p$. Inserting this identity into
the preceding display yields \eqref{eq:b-kernel-bound}.

\smallskip\noindent\emph{(ii) Littlewood--Paley operators and projections.}
Since $\varphi\in\mathcal{S}(\mathbb{R}^n)$, the quantity
$C_\varphi:=\sup_{z}(1+|z|_\infty)^{2n+1}|\varphi(z)|$ is finite, and
$|\varphi(z)|\le C_\varphi(1+t)^{-(2n+1)}$ whenever $|z|_\infty\ge t$. Thus
$\kappa_\Delta$ is nonincreasing and
$\kappa_\Delta(t)\le C_\varphi(1+t)^{-(2n+1)}$ with $2n+1>n+n/p$, so that
$K_\Delta$ is finite by (i). For $g\in L^p_w$ the pairing of $g$ with the
Schwartz function $\varphi_j(x-\cdot)$ is an absolutely convergent integral
(Remark~\ref{rem:single-scale-two-sided}(c)), and $\varphi$ is real-valued
(Definition~\ref{def:littlewood-paley}); hence the tempered distribution
$\Delta_jg=\varphi_j*g$ is the function
$x\mapsto\int_{\mathbb{R}^n}\varphi_j(x-y)\,g(y)\,dy$. Its kernel
$\varphi_j(x-y)=2^{jn}\varphi(2^j(x-y))$ is bounded by
$2^{jn}\kappa_\Delta(2^j|x-y|_\infty)$, and (i) gives
\eqref{eq:b-delta-bound}.

The functions $\chi$ and $\tilde{\chi}$ belong to
$\mathcal{S}(\mathbb{R}^n)$, their Fourier transforms lying in
$C_c^\infty(\mathbb{R}^n)$, so that
$C_\chi:=\sup_z(1+|z|_\infty)^{3n+3}|\chi(z)|$ and
$C_{\tilde{\chi}}:=\sup_z(1+|z|_\infty)^{3n+3}|\tilde{\chi}(z)|$
are finite. Put
$\Pi^+(u,v):=\sum_{k\in\mathbb{Z}^n}|\chi(u-k)|\,|\tilde{\chi}(v-k)|$,
so that $|\Pi|\le\Pi^+$. For $u,v\in\mathbb{R}^n$ and $k\in\mathbb{Z}^n$ one
of the numbers $|u-k|_\infty$, $|v-k|_\infty$ is at least $|u-v|_\infty/2$,
and therefore
\begin{align*}
&(1+|u-k|_\infty)^{-3n-3}(1+|v-k|_\infty)^{-3n-3}\\
&\quad\le\Bigl(1+\tfrac{1}{2}|u-v|_\infty\Bigr)^{-(2n+1)}
\Bigl((1+|u-k|_\infty)^{-(n+2)}+(1+|v-k|_\infty)^{-(n+2)}\Bigr).
\end{align*}
For $z\in\mathbb{R}^n$ and an integer $\nu\ge 0$, at most $(2\nu+2)^n$
points $k\in\mathbb{Z}^n$ satisfy $|z-k|_\infty<\nu+1$; grouping the points
$k$ according to the integer part of $|z-k|_\infty$ gives
$\sum_{k}(1+|z-k|_\infty)^{-(n+2)}\le\sum_{\nu\ge 0}(2\nu+2)^n(1+\nu)^{-(n+2)}
=2^n\zeta(2)$. Together with $(1+t/2)^{-1}\le 2(1+t)^{-1}$ this gives
\[
|\Pi(u,v)|\le\Pi^+(u,v)\le C_P\,(1+|u-v|_\infty)^{-(2n+1)},
\qquad C_P:=2^{3n+2}\,\zeta(2)\,C_\chi C_{\tilde{\chi}} .
\]
Hence $\kappa_P$ and $\kappa_P^+(t):=\sup\{\Pi^+(u,v):|u-v|_\infty\ge t\}$
are nonincreasing and bounded by $C_P(1+t)^{-(2n+1)}$, and $K_P$ is finite
by (i). Both $\Pi$ and $\Pi^+$ are pointwise limits of continuous partial
sums, hence Borel measurable.

Fix $m\in\mathbb{Z}$ and $g\in L^p_w$, and put
$c_k:=\langle g,\tilde{\chi}_{m,k}\rangle$, an absolutely convergent
integral (Remark~\ref{rem:single-scale-two-sided}(c)). Since $|2^mx-2^my|_\infty=2^m|x-y|_\infty$,
the kernels $\mathcal{K}_m(x,y):=2^{mn}\Pi(2^mx,2^my)$ and
$\mathcal{K}_m^+(x,y):=2^{mn}\Pi^+(2^mx,2^my)$ satisfy the hypothesis of
(i) with $j=m$ and with $\kappa=\kappa_P$ and $\kappa=\kappa_P^+$,
respectively. As
$\chi_{m,k}(x)\overline{\tilde{\chi}_{m,k}(y)}
=2^{mn}\chi(2^mx-k)\overline{\tilde{\chi}(2^my-k)}$, Tonelli's
theorem gives
\[
U(x):=\sum_{k\in\mathbb{Z}^n}|c_k|\,|\chi_{m,k}(x)|
\le\int_{\mathbb{R}^n}\mathcal{K}_m^+(x,y)\,|g(y)|\,dy ,
\]
and (i), applied to $|g|$, shows that the right-hand side belongs to
$L^p_w$; in particular it is finite almost everywhere. For $\nu\ge 0$ put
$U_\nu(x):=\sum_{|k|_\infty>\nu}|c_k|\,|\chi_{m,k}(x)|$. Then
$U_\nu\le U$ and $U_\nu\to 0$ almost everywhere, so that
$\|U_\nu\|_{L^p_w}\to 0$ by dominated convergence. Every finite set
$F\subset\{k:|k|_\infty>\nu\}$ satisfies
$|\sum_{k\in F}c_k\chi_{m,k}|\le U_\nu$, hence
$\|\sum_{k\in F}c_k\chi_{m,k}\|_{L^p_w}\le\|U_\nu\|_{L^p_w}$. The
series defining $P_mg$ therefore satisfies the Cauchy criterion for
unconditional convergence, and it converges unconditionally in $L^p_w$. Its partial sums over
$\{|k|_\infty\le\nu\}$ converge in $L^p_w$ to $P_mg$ and, wherever
$U(x)<\infty$, pointwise to $\sum_kc_k\chi_{m,k}(x)$, so the two limits
coincide almost everywhere. Wherever
$\int\mathcal{K}_m^+(x,y)|g(y)|\,dy<\infty$, Fubini's theorem gives
$\sum_kc_k\chi_{m,k}(x)=\int_{\mathbb{R}^n}\mathcal{K}_m(x,y)\,g(y)\,dy$.
Hence $P_mg$ is the operator of (i) with kernel $\mathcal{K}_m$ applied to
$g$, and (i) yields \eqref{eq:b-projection-bounded}.

\smallskip\noindent\emph{(iii) Range and reproduction on $V_m$.} For
$g\in L^p_w$, $P_mg$ is the limit in $L^p_w$ of finite linear combinations of
the functions $\chi_{m,k}$, which belong to $L^p_w$ by the paragraph
following Definition~\ref{def:b-multiresolution}; hence $P_mg\in V_m$. The substitution $u=2^mx$ gives
$\langle\chi_{m,k'},\tilde{\chi}_{m,k}\rangle
=\langle\chi_{0,k'},\tilde{\chi}_{0,k}\rangle=\delta_{k,k'}$ by
Assumption~\ref{ass:mra}, so the series defining $P_m\chi_{m,k'}$ has the single
nonzero term $\chi_{m,k'}$, and $P_m\chi_{m,k'}=\chi_{m,k'}$.
By linearity, $P_mh=h$ for every finite linear combination $h$ of the
functions $\chi_{m,k}$. Every $g\in V_m$ is the limit in $L^p_w$ of such
combinations $h$, and \eqref{eq:b-projection-bounded} gives
$P_mg=\lim P_mh=\lim h=g$. Thus $P_m$ maps $L^p_w$ onto $V_m$ and fixes every
element of $V_m$.

\smallskip\noindent\emph{(iv) Convergence of $P_Jg$ as $J\to\infty$.} We
first show that $P_Jf\to f$ in $L^2(\mathbb{R}^n)$ for every
$f\in L^2(\mathbb{R}^n)$, where $P_Jf$ denotes the right-hand side of
\eqref{eq:b-mra-identity} with $m=J$. By (ii), applied with $p=2$ and
$w\equiv 1$, the operators $P_J$ are bounded on $L^2(\mathbb{R}^n)$ uniformly
in $J$. Let $f=\psi^{(\epsilon')}_{j',k'}$ and $J>j'$. By
\eqref{eq:b-mra-identity} with $m=J$ and the biorthogonality relation
\eqref{eq:biorthogonality},
\[
P_Jf=\sum_{\epsilon=1}^{L}\sum_{j<J}\sum_{k\in\mathbb{Z}^n}
\langle\psi^{(\epsilon')}_{j',k'},\tilde{\psi}^{(\epsilon)}_{j,k}\rangle\,
\psi^{(\epsilon)}_{j,k}=f .
\]
Hence $P_Jf=f$ for all sufficiently large $J$ whenever $f$ is a finite linear
combination of the functions $\psi^{(\epsilon)}_{j,k}$. Such combinations are
dense in $L^2(\mathbb{R}^n)$ (Definition~\ref{def:biorthogonal-riesz}), and
the uniform bound extends the convergence to every $f\in L^2(\mathbb{R}^n)$.

Now let $h\in L^\infty_c$ vanish outside $\{|y|_\infty\le\rho\}$ for
some $\rho>0$. By the
first fact above, the right-hand side of \eqref{eq:b-mra-identity} with
$m=J$ coincides almost everywhere with $P_Jh$, so that $P_Jh\to h$ in
$L^2(\mathbb{R}^n)$. By the representation of $P_J$ in (ii) and the bound on
$\Pi$, for almost every $x$,
\[
|P_Jh(x)|\le C_P\,\|h\|_{L^\infty}\int_{\{|y|_\infty\le\rho\}}
2^{Jn}\bigl(1+2^J|x-y|_\infty\bigr)^{-(2n+1)}\,dy .
\]
The integral does not exceed
$I_n:=\int_{\mathbb{R}^n}(1+|z|_\infty)^{-(2n+1)}\,dz$. If $|x|_\infty\ge 2\rho$
and $J\ge 0$, then $|x-y|_\infty\ge|x|_\infty/2$ on the domain of
integration, and the integrand is at most
$2^{Jn}(2^{J-1}|x|_\infty)^{-(2n+1)}\le 2^{2n+1}|x|_\infty^{-(2n+1)}$. Since
$|x|_\infty\ge 2\rho(1+2\rho)^{-1}(1+|x|_\infty)$ for $|x|_\infty\ge 2\rho$,
it follows that, for every $J\ge 0$, $|P_Jh-h|\le C_hD$ almost everywhere,
where $D(x):=(1+|x|_\infty)^{-(2n+1)}$ and
$C_h:=\|h\|_{L^\infty}(1+2\rho)^{2n+1}
\bigl(1+C_P\max\{I_n,\,2^{2n+1}(2\rho)^{-n-1}\}\bigr)$. The function $D$
belongs to $L^p_w$. Indeed, on $Q^{(i)}_{0,0}\setminus Q^{(i-1)}_{0,0}$,
$i\ge 1$, one has $|x|_\infty\ge i-1$ and hence $D\le i^{-(2n+1)}$, while $D\le 1$ on $Q_{0,0}$.
Inequality \eqref{eq:ap-doubling}, applied to the cube
$Q^{(i)}_{0,0}$ and the set $Q_{0,0}\subset Q^{(i)}_{0,0}$, gives
$w(Q^{(i)}_{0,0})\le[w]_{A_p}(2i+1)^{np}\,w(Q_{0,0})$, and therefore
$\|D\|_{L^p_w}^p\le w(Q_{0,0})\bigl(1+[w]_{A_p}\sum_{i\ge 1}
i^{-(2n+1)p}(2i+1)^{np}\bigr)<\infty$, since
$i^{-(2n+1)p}(2i+1)^{np}\le 3^{np}\,i^{-(n+1)p}$ and $(n+1)p>1$. Given a sequence
$J_\mu\to\infty$, the convergence in $L^2(\mathbb{R}^n)$ provides a
subsequence along which $P_{J_\mu}h\to h$ almost everywhere, and dominated
convergence with the majorant $C_h^pD^pw$ gives
$\|P_{J_\mu}h-h\|_{L^p_w}\to 0$ along that subsequence. Since every sequence
$J_\mu\to\infty$ has such a subsequence, $\|P_Jh-h\|_{L^p_w}\to 0$ as
$J\to\infty$. Finally, for $g\in L^p_w$ and $h\in L^\infty_c$,
\eqref{eq:b-projection-bounded} gives
\[
\|P_Jg-g\|_{L^p_w}
\le\bigl(K_P[w]_{A_p}^{1/p}+1\bigr)\|g-h\|_{L^p_w}+\|P_Jh-h\|_{L^p_w},
\]
so that
$\limsup_{J\to\infty}\|P_Jg-g\|_{L^p_w}\le(K_P[w]_{A_p}^{1/p}+1)\|g-h\|_{L^p_w}$.
The density of $L^\infty_c$ in $L^p_w$ yields \eqref{eq:b-projection-limit}.

\smallskip\noindent\emph{(v) The wavelet blocks.} Fix $j\in\mathbb{Z}$,
$1\le\epsilon\le L$ and $g\in L^p_w$, and put
$c^{(\epsilon)}_k:=\langle g,\tilde{\psi}^{(\epsilon)}_{j,k}\rangle$, an
absolutely convergent integral (Remark~\ref{rem:single-scale-two-sided}(c)). For a finite set
$F\subset\mathbb{Z}^n$, the upper half of \eqref{eq:sse-refined}
with $\Lambda=F$ and $K_U(\psi^{(\epsilon)})\le K_U$ gives
\[
\Bigl\|\sum_{k\in F}c^{(\epsilon)}_k\psi^{(\epsilon)}_{j,k}\Bigr\|_{L^p_w}
\le K_U[w]_{A_p}^{1/p}\,2^{jn/2}
\Bigl(\sum_{k\in F}|c^{(\epsilon)}_k|^p\,w(Q_{j,k})\Bigr)^{1/p},
\]
and \eqref{eq:sse-analysis-general}, with
$K_L(\tilde{\psi}^{(\epsilon)})\le K_L$, bounds the right-hand side with $F$
replaced by $\mathbb{Z}^n$ by $K_UK_L[w]_{A_p}^{2/p}\|g\|_{L^p_w}$. In
particular the right-hand side for $F\subset\{k:|k|_\infty>\nu\}$ tends to
zero as $\nu\to\infty$ uniformly in $F$, so the series over $k$ converges
unconditionally in $L^p_w$, with norm at most
$K_UK_L[w]_{A_p}^{2/p}\|g\|_{L^p_w}$. Summing over the $L$ labels $\epsilon$
by the triangle inequality, we see that the series $r^{(j)}(g)$ converges
unconditionally in $L^p_w$ and that
$\|r^{(j)}(g)\|_{L^p_w}\le L\,K_UK_L[w]_{A_p}^{2/p}\|g\|_{L^p_w}$. Hence
$r^{(j)}$ and, by \eqref{eq:b-projection-bounded}, $P_J-P_m$ are bounded
linear operators on $L^p_w$.

Now let $g\in L^2(\mathbb{R}^n)\cap L^p_w$ and $m<J$. The left-hand side of
\eqref{eq:b-mra-identity} for $m$ is the subseries, over the indices with
$j<m$, of the unconditionally convergent series on the left-hand side for
$J$. Subtracting the two instances of \eqref{eq:b-mra-identity}, we obtain
in $L^2(\mathbb{R}^n)$
\[
\sum_{k\in\mathbb{Z}^n}\langle g,\tilde{\chi}_{J,k}\rangle\,\chi_{J,k}
-\sum_{k\in\mathbb{Z}^n}\langle g,\tilde{\chi}_{m,k}\rangle\,\chi_{m,k}
=\sum_{j=m}^{J-1}\sum_{\epsilon=1}^{L}\sum_{k\in\mathbb{Z}^n}
\langle g,\tilde{\psi}^{(\epsilon)}_{j,k}\rangle\,\psi^{(\epsilon)}_{j,k},
\]
every series converging unconditionally in $L^2(\mathbb{R}^n)$. The same
series converge unconditionally in $L^p_w$, with sums $P_Jg$, $P_mg$ and
$r^{(j)}(g)$, and by the first fact above the sums in the two spaces coincide
almost everywhere. Hence \eqref{eq:b-tail-identity} holds for every
$g\in L^2(\mathbb{R}^n)\cap L^p_w$. Both sides of \eqref{eq:b-tail-identity}
are bounded linear operators on $L^p_w$, and $L^\infty_c$ is dense in
$L^p_w$, so \eqref{eq:b-tail-identity} holds for every $g\in L^p_w$.

\smallskip\noindent\emph{(vi) Annihilation.} The Fourier transform of
$\chi_{m,k}$ is
$\xi\mapsto 2^{-mn/2}e^{-2\pi i\,2^{-m}k\cdot\xi}\,\hat{\chi}(2^{-m}\xi)$,
which vanishes for $|\xi|>2^mR$, whereas
$\hat{\varphi}_{j'}(\xi)=\hat{\varphi}(2^{-j'}\xi)$ vanishes for
$|\xi|<c_12^{j'}$. If $j'\ge m+M_\chi+1$, then $2^{M_\chi}\ge R/c_1$
gives $c_12^{j'}\ge 2^{m+1}R>2^mR$, so that
$\hat{\varphi}_{j'}\hat{\chi}_{m,k}\equiv 0$ and
$\Delta_{j'}\chi_{m,k}=\varphi_{j'}*\chi_{m,k}=0$. By linearity,
$\Delta_{j'}h=0$ for every finite linear combination $h$ of the functions
$\chi_{m,k}$. Every $g\in V_m$ is the limit in $L^p_w$ of such
combinations, and \eqref{eq:b-delta-bound} gives $\Delta_{j'}g=0$ in $L^p_w$.
\end{proof}

\begin{theorem}[Single-scale Bernstein estimate]\label{thm:B1}
	Let $s > 0$, $1 < p < \infty$, $0 < q < \infty$, and $w \in A_p$. Let
	$j \in \mathbb{Z}$, let $\Lambda \subset \mathbb{Z}^n$ be finite, and let
	$r := \sum_{\epsilon=1}^{L} \sum_{k \in \Lambda} c^{(\epsilon)}_{k}\,
	\psi^{(\epsilon)}_{j,k}$ with $c^{(\epsilon)}_{k} \in \mathbb{C}$. Then
	\begin{equation}\label{eq:B1-lower}
		C_M^{-1}\, 2^{js}\, \|r\|_{L^p_w}
		\le \|r\|_{\dot{B}^s_{p,q}(w)}
	\end{equation}
	and
	\begin{equation}\label{eq:B1-upper}
		\|r\|_{\dot{B}^s_{p,q}(w)}
		\le C'_M\, K_\Delta\, [w]_{A_p}^{1/p}\, 2^{js}\, \|r\|_{L^p_w},
	\end{equation}
	with $C_M$, $C'_M$ as in \eqref{eq:finite-band} and $K_\Delta$
	as in Lemma~\ref{lem:Bern:projection-bounded}(ii).
\end{theorem}

\begin{proof}[Proof of Theorem~\ref{thm:B1}.]
By Assumption~\ref{ass:wavelet-system}, every $\hat{\psi}^{(\epsilon)}$
belongs to $C_c^\infty(\mathbb{R}^n)$ and is supported in
$\{\xi : a \le |\xi| \le b\}$, the radii $a, b$ being common to all labels;
in particular $\psi^{(\epsilon)} \in \mathcal{S}(\mathbb{R}^n)$, and hence
$r \in \mathcal{S}(\mathbb{R}^n) \subset L^p_w$ (see the paragraph following Definition~\ref{def:b-multiresolution}). The
substitution $y = 2^j x - k$ gives
$|\hat{\psi}^{(\epsilon)}_{j,k}(\xi)|
= 2^{-jn/2} |\hat{\psi}^{(\epsilon)}(2^{-j}\xi)|$, which vanishes unless
$2^j a \le |\xi| \le 2^j b$; thus $\hat{r}$ is supported in the single
annulus $\{\xi : 2^j a \le |\xi| \le 2^j b\}$. The finite-band property
\eqref{eq:finite-band}, applied to $h = r$, gives
\[
\Delta_{j'} r = 0 \quad \text{for } |j' - j| > M,
\qquad
r = \sum_{|i| \le M} \Delta_{j+i} r,
\]
where $\Delta_{j'} r = \varphi_{j'} * r \in \mathcal{S}(\mathbb{R}^n)$, so
that the second relation is an identity between Schwartz functions. Since
each $\Delta_{j'}$ annihilates polynomials
(Definition~\ref{def:littlewood-paley}), the functions $\Delta_{j'} r$ are
also the Littlewood--Paley blocks of the class of $r$ modulo $\mathcal{P}$,
and $\|r\|_{\dot{B}^s_{p,q}(w)}$ is the quasi-norm \eqref{eq:weighted-besov-norm}
of this class.

A single term of \eqref{eq:weighted-besov-norm} does not exceed the whole sum, so
$2^{j's} \|\Delta_{j'} r\|_{L^p_w} \le \|r\|_{\dot{B}^s_{p,q}(w)}$ for every
$j' \in \mathbb{Z}$. The triangle inequality in $L^p_w$ therefore gives
\begin{align*}
2^{js} \|r\|_{L^p_w}
&\le \sum_{|i| \le M} 2^{-is}\, 2^{(j+i)s}\, \|\Delta_{j+i} r\|_{L^p_w}\\
&\le \Bigl( \sum_{|i| \le M} 2^{-is} \Bigr) \|r\|_{\dot{B}^s_{p,q}(w)}
= C_M\, \|r\|_{\dot{B}^s_{p,q}(w)},
\end{align*}
the last equality following from the substitution $i \mapsto -i$. Dividing
by $C_M > 0$ yields \eqref{eq:B1-lower}.

For \eqref{eq:B1-upper}, only the indices $j' = j + i$ with $|i| \le M$
contribute to \eqref{eq:weighted-besov-norm}, and \eqref{eq:b-delta-bound},
applied with $g = r \in L^p_w$ and with $j + i$ in place of $j$, bounds each
contribution. Hence
\[
\|r\|_{\dot{B}^s_{p,q}(w)}^q
= \sum_{|i| \le M} 2^{(j+i)sq}\, \|\Delta_{j+i} r\|_{L^p_w}^q
\le \Bigl( \sum_{|i| \le M} 2^{isq} \Bigr)
\bigl( K_\Delta\, [w]_{A_p}^{1/p}\, 2^{js}\, \|r\|_{L^p_w} \bigr)^q .
\]
The sum on the right equals $(C'_M)^q$, and taking $q$-th roots yields
\eqref{eq:B1-upper}.
\end{proof}

No power of $[w]_{A_p}$ enters \eqref{eq:B1-lower}, which rests on the
finite band of \eqref{eq:finite-band}, the definition
\eqref{eq:weighted-besov-norm} and the triangle inequality in $L^p_w$, whereas
the power $[w]_{A_p}^{1/p}$ in \eqref{eq:B1-upper} is that of
\eqref{eq:b-delta-bound}.

	The upper estimate has a counterpart on a multiresolution subspace.
	Theorem~\ref{thm:B2} is the Bernstein inequality on the subspace $V_m$.

	\begin{theorem}[Bernstein inequality on multiresolution subspaces]
		\label{thm:B2}
		Let $s > 0$, $1 < p < \infty$, $0 < q < \infty$, $w \in A_p$, and
		$m \in \mathbb{Z}$, and let $V_m$ be as in \eqref{eq:b-mra-space}. For
		every $g \in V_m$,
		\begin{equation}\label{eq:B2}
			\|g\|_{\dot{B}^s_{p,q}(w)}
			\le K_\Delta\, 2^{M_\chi s}\, (1 - 2^{-sq})^{-1/q}\,
			[w]_{A_p}^{1/p}\, 2^{ms}\, \|g\|_{L^p_w},
		\end{equation}
		with $K_\Delta$ and $M_\chi$ as in
		Lemma~\ref{lem:Bern:projection-bounded}.
	\end{theorem}

\begin{proof}[Proof of Theorem~\ref{thm:B2}.]
Let $g \in V_m$. Since $V_m \subset L^p_w$, the function $g$ defines a
tempered distribution (Remark~\ref{rem:single-scale-two-sided}(c)), and for every
$j' \in \mathbb{Z}$ the block $\Delta_{j'} g = \varphi_{j'} * g$ is the
function of Lemma~\ref{lem:Bern:projection-bounded}(ii). Since each
$\Delta_{j'}$ annihilates polynomials
(Definition~\ref{def:littlewood-paley}), these functions are also the
Littlewood--Paley blocks of the class of $g$ modulo $\mathcal{P}$, and
$\|g\|_{\dot{B}^s_{p,q}(w)}$ is the quasi-norm \eqref{eq:weighted-besov-norm} of
this class.

By \eqref{eq:b-annihilation}, $\Delta_{j'} g = 0$ for every
$j' \ge m + M_\chi + 1$, so only the indices $j' \le m + M_\chi$
contribute to \eqref{eq:weighted-besov-norm}. For each of these indices,
\eqref{eq:b-delta-bound}, applied with $j'$ in place of $j$, gives
$\|\Delta_{j'} g\|_{L^p_w} \le K_\Delta\, [w]_{A_p}^{1/p}\, \|g\|_{L^p_w}$.
Hence
\[
\|g\|_{\dot{B}^s_{p,q}(w)}^q
= \sum_{j' \le m + M_\chi} 2^{j'sq}\, \|\Delta_{j'} g\|_{L^p_w}^q
\le \Bigl( \sum_{j' \le m + M_\chi} 2^{j'sq} \Bigr)
\bigl( K_\Delta\, [w]_{A_p}^{1/p}\, \|g\|_{L^p_w} \bigr)^q .
\]
The substitution $j' = m + M_\chi - i$, $i \ge 0$, together with
$sq > 0$, gives
\[
\sum_{j' \le m + M_\chi} 2^{j'sq}
= 2^{(m + M_\chi)sq} \sum_{i \ge 0} 2^{-isq}
= 2^{(m + M_\chi)sq}\, (1 - 2^{-sq})^{-1},
\]
and taking $q$-th roots yields \eqref{eq:B2}.
\end{proof}

An inverse estimate is finally paired with a direct one to identify the
approximation scale. Lemma~\ref{lem:Bern:linear-jackson} is the linear
Jackson estimate on $V_m$, Theorem~\ref{thm:B3} the identification of the
linear approximation scale with the weighted smoothness class, and
Remark~\ref{rem:Bern:B3-linear} the comparison of this linear
identification with its nonlinear counterpart.

	\begin{lemma}[Linear Jackson estimate on multiresolution subspaces]
		\label{lem:Bern:linear-jackson}
		Let $s > 0$, $1 < p < \infty$, $0 < q < \infty$, and $w \in A_p$. For
		every $f \in L^p_w \cap \dot{B}^s_{p,q}(w)$ and every
		$m \in \mathbb{Z}$,
		\begin{equation}\label{eq:b-linear-jackson}
			E(f, V_m)_{L^p_w}
			\le \frac{L\, C_M\, K_U\, K_L}{1 - 2^{-s}}\,
			[w]_{A_p}^{2/p}\, 2^{-ms}\, \|f\|_{\dot{B}^s_{p,q}(w)},
		\end{equation}
		with $C_M$ as in \eqref{eq:finite-band} and $K_U$, $K_L$ as
		in \eqref{eq:constant-form}.
	\end{lemma}

\begin{proof}[Proof of Lemma~\ref{lem:Bern:linear-jackson}.]
Let $f \in L^p_w \cap \dot{B}^s_{p,q}(w)$ and $m \in \mathbb{Z}$. By
Lemma~\ref{lem:Bern:projection-bounded}(iii), $P_m f \in V_m$, and therefore
\[
E(f, V_m)_{L^p_w} \le \|f - P_m f\|_{L^p_w} .
\]
We first bound the wavelet blocks $r^{(j)}(f)$ of
Lemma~\ref{lem:Bern:projection-bounded}(v) one generation at a time. Fix
$j \in \mathbb{Z}$, put
$c^{(\epsilon)}_k := \langle f, \tilde{\psi}^{(\epsilon)}_{j,k} \rangle$, an
absolutely convergent integral (Remark~\ref{rem:single-scale-two-sided}(c)), and let
\[
S^{(\epsilon)}_j := 2^{jn/2} \Bigl( \sum_{k \in \mathbb{Z}^n}
|c^{(\epsilon)}_k|^p\, w(Q_{j,k}) \Bigr)^{1/p},
\qquad 1 \le \epsilon \le L .
\]
The quantity $S_j(f)$ of \eqref{eq:Sj-def} equals
$\bigl( \sum_{\epsilon=1}^{L} (S^{(\epsilon)}_j)^p \bigr)^{1/p}$, and
\eqref{eq:coefficient-bound} gives
\[
S_j(f) \le L^{1/p}\, C_M\, K_L\, [w]_{A_p}^{1/p}\, 2^{-js}\,
\|f\|_{\dot{B}^s_{p,q}(w)} < \infty .
\]
For $\nu \ge 0$ let $F_\nu := \{ k \in \mathbb{Z}^n : |k|_\infty \le \nu \}$.
For each label $\epsilon$, the upper half of
\eqref{eq:sse-refined}, applied to the pair
$(\psi^{(\epsilon)}, \tilde{\psi}^{(\epsilon)})$ at generation $j$ with
$\Lambda = F_\nu$, together with $K_U(\psi^{(\epsilon)}) \le K_U$ and the
enlargement of the sum over $F_\nu$ on its right-hand side to the sum over
$\mathbb{Z}^n$, gives
\[
\Bigl\| \sum_{k \in F_\nu} c^{(\epsilon)}_k\, \psi^{(\epsilon)}_{j,k}
\Bigr\|_{L^p_w}
\le K_U\, [w]_{A_p}^{1/p}\, S^{(\epsilon)}_j .
\]
The triangle inequality in $L^p_w$ over the $L$ labels and H\"older's
inequality in $\epsilon$ with the exponents $p'$ and $p$ give
\[
\Bigl\| \sum_{\epsilon=1}^{L} \sum_{k \in F_\nu} c^{(\epsilon)}_k\,
\psi^{(\epsilon)}_{j,k} \Bigr\|_{L^p_w}
\le K_U\, [w]_{A_p}^{1/p} \sum_{\epsilon=1}^{L} S^{(\epsilon)}_j
\le L^{1/p'}\, K_U\, [w]_{A_p}^{1/p}\, S_j(f) .
\]
The function on the left-hand side is the partial sum of the series
$r^{(j)}(f)$ over the finite set $\{1, \dots, L\} \times F_\nu$. These sets
increase with $\nu$ and exhaust $\{1, \dots, L\} \times \mathbb{Z}^n$, and
the series converges unconditionally in $L^p_w$ by
Lemma~\ref{lem:Bern:projection-bounded}(v), so these partial sums converge
to $r^{(j)}(f)$ in $L^p_w$ as $\nu \to \infty$. Letting $\nu \to \infty$,
inserting the bound for $S_j(f)$ and using $L^{1/p'} L^{1/p} = L$, we obtain
\[
\|r^{(j)}(f)\|_{L^p_w}
\le L\, C_M\, K_U\, K_L\, [w]_{A_p}^{2/p}\, 2^{-js}\,
\|f\|_{\dot{B}^s_{p,q}(w)},
\qquad j \in \mathbb{Z},
\]
where the factor $[w]_{A_p}^{2/p}$ is the product of the factor
$[w]_{A_p}^{1/p}$ of the upper half of \eqref{eq:sse-refined} and
the factor $[w]_{A_p}^{1/p}$ of \eqref{eq:coefficient-bound}.

It remains to sum over the generations $j \ge m$. Since $s > 0$,
$\sum_{j \ge m} 2^{-js} = 2^{-ms} (1 - 2^{-s})^{-1}$. For every integer
$J > m$, \eqref{eq:b-tail-identity}, the triangle inequality in $L^p_w$ and
the preceding bound give
\begin{align*}
\|f - P_m f\|_{L^p_w}
&\le \|f - P_J f\|_{L^p_w} + \sum_{j=m}^{J-1} \|r^{(j)}(f)\|_{L^p_w}\\
&\le \|f - P_J f\|_{L^p_w}
+ \frac{L\, C_M\, K_U\, K_L}{1 - 2^{-s}}\, [w]_{A_p}^{2/p}\, 2^{-ms}\,
\|f\|_{\dot{B}^s_{p,q}(w)} .
\end{align*}
By \eqref{eq:b-projection-limit}, $\|f - P_J f\|_{L^p_w} \to 0$ as
$J \to \infty$. Letting $J \to \infty$ and combining the result with the
first display yields \eqref{eq:b-linear-jackson}.
\end{proof}

	\begin{theorem}[Identification of the linear approximation scale]
		\label{thm:B3}
		Let $s > 0$, $1 < p < \infty$, $0 < q < \infty$, and $w \in A_p$. For
		$f \in L^p_w$ put
		\begin{equation}\label{eq:B3-Anorm}
			\begin{aligned}
				\| f \|_{\mathcal{A}^{s/n}_q(L^p_w, \{V_m\})}
				&:= \| f \|_{L^p_w} + \Lambda(f),\\
				\Lambda(f) &:= \left( \sum_{m \ge 0}
				\bigl[\, 2^{ms}\, E(f, V_m)_{L^p_w} \,\bigr]^q \right)^{1/q},
			\end{aligned}
		\end{equation}
		following the convention of \cite[Ch.~7, \S 9]{DeVore1993}, and let
		$\mathcal{A}^{s/n}_q(L^p_w, \{V_m\})$ be the set of $f \in L^p_w$ with
		$\Lambda(f) < \infty$. Then
		\begin{equation}\label{eq:B3-equiv}
			\begin{aligned}
				\Lambda(f) &\le C_{\mathrm{dir}}\, [w]_{A_p}^{2/p}\,
				\| f \|_{\dot{B}^s_{p,q}(w)},\\
				\| f \|_{\dot{B}^s_{p,q}(w)} &\le C_{\mathrm{inv}}\,
				[w]_{A_p}^{1/p} \bigl( \| f \|_{L^p_w} + \Lambda(f) \bigr),
			\end{aligned}
		\end{equation}
		the first inequality for every $f \in L^p_w \cap \dot{B}^s_{p,q}(w)$
		and the second for every $f \in L^p_w$, where
		\[
		C_{\mathrm{dir}} := L\, K_U\, K_L\, C_H\, C''_M,
		\qquad
		C_{\mathrm{inv}} := 2^{1/q + 1 + (M_\chi + 1)s}\, (1 - 2^{-sq})^{-1/q}\,
		K_\Delta,
		\]
		with $C_H := (1 - 2^{-s})^{-1}$ and $C''_M := C_M$ for $q \ge 1$, and
		$C_H := (1 - 2^{-sq})^{-1/q}$ and $C''_M := C'_M$ for $0 < q < 1$.
		Consequently
		$\mathcal{A}^{s/n}_q(L^p_w, \{V_m\}) = L^p_w \cap \dot{B}^s_{p,q}(w)$,
		and $\|f\|_{\mathcal{A}^{s/n}_q(L^p_w, \{V_m\})}$ is equivalent on this set to
		$\|f\|_{L^p_w} + \|f\|_{\dot{B}^s_{p,q}(w)}$.
	\end{theorem}

\begin{proof}[Proof of Theorem~\ref{thm:B3}.]
Both inequalities in \eqref{eq:B3-equiv} are reduced to the
Littlewood--Paley blocks of $f$: the first through the wavelet blocks
$r^{(j)}(f)$ of Lemma~\ref{lem:Bern:projection-bounded}(v), and the second
through the annihilation property \eqref{eq:b-annihilation}.

Let $f \in L^p_w$ and write $E_m := E(f, V_m)_{L^p_w}$; since $0 \in V_m$,
$E_m \le \|f\|_{L^p_w} < \infty$. The function $f$ defines a tempered
distribution (Remark~\ref{rem:single-scale-two-sided}(c)), and for every $j' \in \mathbb{Z}$ the
block $\Delta_{j'} f = \varphi_{j'} * f$ is the function of
Lemma~\ref{lem:Bern:projection-bounded}(ii). Since each $\Delta_{j'}$
annihilates polynomials (Definition~\ref{def:littlewood-paley}), these
functions are also the Littlewood--Paley blocks of the class of $f$ modulo
$\mathcal{P}$, and $\|f\|_{\dot{B}^s_{p,q}(w)}$ is the value in
$[0, \infty]$ of the right-hand side of \eqref{eq:weighted-besov-norm} for this
class; it is finite exactly when $f \in L^p_w \cap \dot{B}^s_{p,q}(w)$. With
$\eta_{j'} := 2^{j's} \|\Delta_{j'} f\|_{L^p_w}$ we thus have
$\|f\|_{\dot{B}^s_{p,q}(w)} = ( \sum_{j' \in \mathbb{Z}} \eta_{j'}^q )^{1/q}$.

We first prove the first inequality in \eqref{eq:B3-equiv}, and assume
$f \in L^p_w \cap \dot{B}^s_{p,q}(w)$. Fix $j \in \mathbb{Z}$ and
$1 \le \epsilon \le L$. The coefficients
$c^{(\epsilon)}_k := \langle f, \tilde{\psi}^{(\epsilon)}_{j,k} \rangle$,
$k \in \mathbb{Z}^n$, are absolutely convergent integrals
(Remark~\ref{rem:single-scale-two-sided}(c)), and by \eqref{eq:coef-identity}
$c^{(\epsilon)}_k = \sum_{|i| \le M}
\langle \Delta_{j+i} f, \tilde{\psi}^{(\epsilon)}_{j,k} \rangle$, where each
$\Delta_{j+i} f$ belongs to $L^p_w$ because a single term of
\eqref{eq:weighted-besov-norm} does not exceed the whole sum. Minkowski's
inequality in the weighted $\ell^p(\mathbb{Z}^n)$ norm over the $2M+1$
indices $i$, followed by \eqref{eq:sse-analysis-general} for the dual
generator $\tilde{\psi}^{(\epsilon)}$ with $g = \Delta_{j+i} f$ and by
$K_L(\tilde{\psi}^{(\epsilon)}) \le K_L$, gives
\[
S^{(\epsilon)}_j := 2^{jn/2} \Bigl( \sum_{k \in \mathbb{Z}^n}
|c^{(\epsilon)}_k|^p\, w(Q_{j,k}) \Bigr)^{1/p}
\le K_L\, [w]_{A_p}^{1/p} \sum_{|i| \le M} \|\Delta_{j+i} f\|_{L^p_w} .
\]
For $\nu \ge 0$ let $F_\nu := \{ k \in \mathbb{Z}^n : |k|_\infty \le \nu \}$.
The upper half of \eqref{eq:sse-refined}, applied to the pair
$(\psi^{(\epsilon)}, \tilde{\psi}^{(\epsilon)})$ at generation $j$ with the
finite index set $F_\nu$, together with $K_U(\psi^{(\epsilon)}) \le K_U$ and
the enlargement of the sum over $F_\nu$ on its right-hand side to the sum
over $\mathbb{Z}^n$, bounds the norm of
$\sum_{k \in F_\nu} c^{(\epsilon)}_k \psi^{(\epsilon)}_{j,k}$ in $L^p_w$ by
$K_U [w]_{A_p}^{1/p} S^{(\epsilon)}_j$. By the triangle inequality over the
$L$ labels, the partial sum of the series $r^{(j)}(f)$ over
$\{1, \dots, L\} \times F_\nu$ has norm at most
$K_U [w]_{A_p}^{1/p} \sum_{\epsilon=1}^{L} S^{(\epsilon)}_j$. These finite
sets increase with $\nu$ and exhaust $\{1, \dots, L\} \times \mathbb{Z}^n$,
and the series converges unconditionally in $L^p_w$ by
Lemma~\ref{lem:Bern:projection-bounded}(v). Letting $\nu \to \infty$,
inserting the bound for $S^{(\epsilon)}_j$ for each label and using
$2^{js} \|\Delta_{j+i} f\|_{L^p_w} = 2^{-is} \eta_{j+i}$, we obtain
\[
2^{js}\, \|r^{(j)}(f)\|_{L^p_w}
\le L\, K_U\, K_L\, [w]_{A_p}^{2/p}\, \bar{\eta}_j,
\qquad
\bar{\eta}_j := \sum_{|i| \le M} 2^{-is}\, \eta_{j+i},
\qquad j \in \mathbb{Z} .
\]

Now let $m \ge 0$. By Lemma~\ref{lem:Bern:projection-bounded}(iii),
$P_m f \in V_m$. For every integer $J > m$, \eqref{eq:b-tail-identity} and
the triangle inequality give
$\|f - P_m f\|_{L^p_w} \le \|f - P_J f\|_{L^p_w}
+ \sum_{j=m}^{J-1} \|r^{(j)}(f)\|_{L^p_w}$, and letting $J \to \infty$ by
\eqref{eq:b-projection-limit} we find
$E_m \le \|f - P_m f\|_{L^p_w} \le \sum_{j \ge m} \|r^{(j)}(f)\|_{L^p_w}$.
Multiplying by $2^{ms}$, writing $j = m + l$ and inserting the preceding
bound, we arrive at
\[
2^{ms}\, E_m
\le L\, K_U\, K_L\, [w]_{A_p}^{2/p} \sum_{l \ge 0} 2^{-ls}\, \bar{\eta}_{m+l},
\qquad m \ge 0 .
\]
It remains to sum in $\ell^q$ over $m \ge 0$. We claim that
\begin{equation}\label{eq:b-hardy}
\Bigl( \sum_{m \ge 0} \Bigl( \sum_{l \ge 0} 2^{-ls}\, \bar{\eta}_{m+l} \Bigr)^q
\Bigr)^{1/q}
\le C_H \Bigl( \sum_{j \ge 0} \bar{\eta}_j^q \Bigr)^{1/q} .
\end{equation}
For $q \ge 1$, Minkowski's inequality in $\ell^q$, applied to the series
over $l$, bounds the left-hand side by
$\sum_{l \ge 0} 2^{-ls} ( \sum_{m \ge 0} \bar{\eta}_{m+l}^q )^{1/q}
\le (1 - 2^{-s})^{-1} ( \sum_{j \ge 0} \bar{\eta}_j^q )^{1/q}$. For $0 < q < 1$, the
inequality $( \sum_l x_l )^q \le \sum_l x_l^q$ for nonnegative numbers $x_l$
bounds the $q$-th power of the left-hand side by
$\sum_{l \ge 0} 2^{-lsq} \sum_{m \ge 0} \bar{\eta}_{m+l}^q
\le (1 - 2^{-sq})^{-1} \sum_{j \ge 0} \bar{\eta}_j^q$. Next,
\begin{equation}\label{eq:b-coef-lq}
\Bigl( \sum_{j \ge 0} \bar{\eta}_j^q \Bigr)^{1/q}
\le C''_M\, \|f\|_{\dot{B}^s_{p,q}(w)} .
\end{equation}
Indeed, for $q \ge 1$, Minkowski's inequality in $\ell^q$ over the $2M+1$
indices $i$ and $( \sum_{j \ge 0} \eta_{j+i}^q )^{1/q}
\le \|f\|_{\dot{B}^s_{p,q}(w)}$ give the factor
$\sum_{|i| \le M} 2^{-is} = C_M$; for $0 < q < 1$, the inequality
$( \sum_l x_l )^q \le \sum_l x_l^q$, applied to the sum over $i$, gives
$\sum_{j \ge 0} \bar{\eta}_j^q \le \sum_{|i| \le M} 2^{-isq}
\sum_{j \ge 0} \eta_{j+i}^q \le \sum_{|i| \le M} 2^{-isq}
\|f\|_{\dot{B}^s_{p,q}(w)}^q$, and $\sum_{|i| \le M} 2^{-isq}
= \sum_{|i| \le M} 2^{isq} = (C'_M)^q$ by the substitution $i \mapsto -i$.
Taking the $\ell^q$ quasi-norm over $m \ge 0$ in the bound for
$2^{ms} E_m$ and applying \eqref{eq:b-hardy} and \eqref{eq:b-coef-lq}, we
obtain
\[
\Lambda(f)
\le L\, K_U\, K_L\, C_H\, C''_M\, [w]_{A_p}^{2/p}\,
\|f\|_{\dot{B}^s_{p,q}(w)}
= C_{\mathrm{dir}}\, [w]_{A_p}^{2/p}\, \|f\|_{\dot{B}^s_{p,q}(w)},
\]
the first inequality in \eqref{eq:B3-equiv}; its factor $[w]_{A_p}^{2/p}$
is the product of the factor $[w]_{A_p}^{1/p}$ of the upper half of
\eqref{eq:sse-refined} and that of
\eqref{eq:sse-analysis-general}.

We turn to the second inequality, for an arbitrary $f \in L^p_w$. For each
integer $m \ge 0$ choose $g_m \in V_m$ with
$\|f - g_m\|_{L^p_w} \le 2 E_m$: if $E_m > 0$, such an element exists by
\eqref{eq:b-linear-error}; if $E_m = 0$, then $f$ lies in the closure of
$V_m$ in $L^p_w$, which is $V_m$, and we put $g_m := f$. Let
$j' \ge M_\chi + 1$ and $m := j' - M_\chi - 1 \ge 0$. Since
$j' = m + M_\chi + 1$, \eqref{eq:b-annihilation} gives
$\Delta_{j'} g_m = 0$, hence
$\Delta_{j'} f = \Delta_{j'} (f - g_m)$, and \eqref{eq:b-delta-bound},
applied with $j'$ in place of $j$ and with $f - g_m$ in place of $g$, gives
$\eta_{j'} \le 2^{j's} K_\Delta [w]_{A_p}^{1/p} \|f - g_m\|_{L^p_w}
\le K_\Delta [w]_{A_p}^{1/p}\, 2^{1 + (M_\chi + 1)s}\, 2^{ms} E_m$.
For $j' \le M_\chi$, \eqref{eq:b-delta-bound}, applied with $j'$ in
place of $j$ and with $g = f$, gives
$\eta_{j'} \le K_\Delta [w]_{A_p}^{1/p}\, 2^{j's} \|f\|_{L^p_w}$. Thus
$\eta_{j'} \le K_\Delta [w]_{A_p}^{1/p} d_{j'}$ for every $j' \in \mathbb{Z}$,
where $d_{j'} := 2^{j's} \|f\|_{L^p_w}$ for $j' \le M_\chi$ and
$d_{j'} := 2^{1 + (M_\chi + 1)s}\, 2^{ms} E_m$ for
$j' = m + M_\chi + 1$, $m \ge 0$. The substitution
$j' = M_\chi - i$, $i \ge 0$, and $sq > 0$ give
\[
\sum_{j' \le M_\chi} d_{j'}^q
= 2^{M_\chi sq} (1 - 2^{-sq})^{-1} \|f\|_{L^p_w}^q,
\qquad
\sum_{j' \ge M_\chi + 1} d_{j'}^q
= 2^{(1 + (M_\chi + 1)s)q}\, \Lambda(f)^q .
\]
With $c_\ast := 2^{1 + (M_\chi + 1)s} (1 - 2^{-sq})^{-1/q}$, the
inequalities $2^{M_\chi s} \le 2^{1 + (M_\chi + 1)s}$ and
$(1 - 2^{-sq})^{-1} \ge 1$ bound the two sums by
$c_\ast^q \|f\|_{L^p_w}^q$ and $c_\ast^q \Lambda(f)^q$, and
$u^q + v^q \le 2 (u + v)^q$ for $u, v \ge 0$. Hence
\[
\|f\|_{\dot{B}^s_{p,q}(w)}
\le K_\Delta\, [w]_{A_p}^{1/p}
\Bigl( \sum_{j' \in \mathbb{Z}} d_{j'}^q \Bigr)^{1/q}
\le 2^{1/q}\, c_\ast\, K_\Delta\, [w]_{A_p}^{1/p}
\bigl( \|f\|_{L^p_w} + \Lambda(f) \bigr),
\]
both sides lying in $[0, \infty]$. Since $2^{1/q} c_\ast K_\Delta
= C_{\mathrm{inv}}$, this is the second inequality in \eqref{eq:B3-equiv};
its factor $[w]_{A_p}^{1/p}$ is that of \eqref{eq:b-delta-bound}.

Finally, if $f \in L^p_w \cap \dot{B}^s_{p,q}(w)$, the first inequality
gives $\Lambda(f) < \infty$; if $\Lambda(f) < \infty$, the second gives
$\|f\|_{\dot{B}^s_{p,q}(w)} < \infty$, so that the class of $f$ modulo
$\mathcal{P}$ belongs to $\dot{B}^s_{p,q}(w)$. Hence
$\mathcal{A}^{s/n}_q(L^p_w, \{V_m\}) = L^p_w \cap \dot{B}^s_{p,q}(w)$, and on
this set \eqref{eq:B3-equiv} bounds $\|f\|_{L^p_w} + \Lambda(f)$ by
$\max\{1, C_{\mathrm{dir}} [w]_{A_p}^{2/p}\}
( \|f\|_{L^p_w} + \|f\|_{\dot{B}^s_{p,q}(w)} )$ and
$\|f\|_{L^p_w} + \|f\|_{\dot{B}^s_{p,q}(w)}$ by
$(1 + C_{\mathrm{inv}} [w]_{A_p}^{1/p}) ( \|f\|_{L^p_w} + \Lambda(f) )$,
which is the equivalence of the two quasi-norms.
\end{proof}

\begin{remark}[Linear versus nonlinear identification]
	\label{rem:Bern:B3-linear}
	The identification of Theorem~\ref{thm:B3} is the linear one: the approximating
	family is the multiresolution sequence $\{V_m\}$, and the smoothness
	scale is the same-exponent space $\dot{B}^s_{p,q}(w)$. Its nonlinear counterpart
	replaces $\{V_m\}$ by the family $\Sigma_N$ of linear combinations of at most
	$N$ elements of the wavelet system, whose selection across scales is free; the matched
	smoothness scale is then expected to lie on the critical line $1/\tau = s/n + 1/p$.
	For $w \in A_p$, a condition on that line can be formulated through the
	coefficients $c^{(\epsilon)}_{j,k} := \langle f, \tilde{\psi}^{(\epsilon)}_{j,k} \rangle$
	of $f \in L^p_w$, as the $\ell^\tau$ summability of the quantities
	$2^{jn/2}\, |c^{(\epsilon)}_{j,k}|\, w(Q_{j,k})^{1/p}$, and not through a space
	$\dot{B}^s_{\tau,\tau}(w)$: Definition~\ref{def:weighted-besov} provides such a
	space only when $\tau > 1$ and $w \in A_\tau$, the first requirement fails for
	$s \ge n/p'$, and for $1 < \tau < p$ the second is not implied by $w \in A_p$.
	The change of summation index
	$\tau < p$ is the structural cost of nonlinear selection. Three
	developments are left to separate directions: the matching lower estimates for
	the weight powers carried on the two sides of \eqref{eq:B3-equiv}, by an
	extremal family adapted to the Muckenhoupt characteristic; the nonlinear
	$\Sigma_N$ strengthening on the critical line above; and the Marchaud-type
	companion to the present direct--inverse pair.
\end{remark}

\bibliographystyle{amsplain}
\bibliography{BS}
\end{document}